\documentclass[twoside,11pt]{article}

\usepackage{jmlr2e}

\renewcommand{\c}{\mathbf{c}}

\newcommand{\beqn}{\begin{eqnarray}}
\newcommand{\beqnn}{\begin{eqnarray*}}
\newcommand{\eeqn}{\end{eqnarray}}
\newcommand{\eeqnn}{\end{eqnarray*}}
\newcommand{\GG}{\mathcal{G}}
 
\def \FF {\mathcal{F}}

\def \FF {\mathcal{F}}
\def \no {\Arrowvert}
\def \ind {\hbox{ 1\hskip -3pt I}}
\def \R {\mathbb{R}}

\def \P  {\mathbb{P}} 
\def \R {\mathbb{R}}
\def \E {\mathbb{E}}
\def \N {\mathbb{N}}

\ShortHeadings{Noisy quantization}{S. Loustau}
\firstpageno{1}

\begin{document}

\title{Anisotropic oracle inequalities\\in noisy quantization}

\author{\name S\'ebastien Loustau \email loustau@math.univ-angers.fr \\
       \addr LAREMA\\
       Universit\'e d'Angers\\
       2 Boulevard Lavoisier,\\
       49045 Angers Cedex, France}

\editor{}

\maketitle

\begin{abstract}
 The effect of errors in variables in quantization is investigated. We prove general exact and non-exact oracle inequalities with fast rates for an empirical minimization based on a noisy sample $Z_i=X_i+\epsilon_i,i=1,\ldots,n$, where $X_i$ are i.i.d. with density $f$ and $\epsilon_i$ are i.i.d. with density $\eta$. These rates depend on the geometry of the density $f$ and the asymptotic behaviour of the characteristic function of $\eta$.
 
This general study can be applied to the problem of $k$-means clustering with noisy data. For this purpose, we introduce a deconvolution $k$-means stochastic minimization which reaches fast rates of convergence under standard Pollard's regularity assumptions.
\end{abstract}


\section{Introduction}
        The goal of empirical vector quantization (\cite{graf}) or clustering (\cite{hartigan75}) is to replace data by an efficient and compact representation, which allows one to reconstruct the original observations with a certain accuracy. The problem was originated in signal processing and has many applications in cluster analysis or information theory. The statistical model could be described as follows. Given independent and identically distributed (i.i.d.) random variables $X_1,\ldots ,X_n$, with unknown law $P$ with density $f$ on $\R^d$ with respect to the Lebesgue measure, we want to choose a quantizer (or classifier) $g\in\mathcal{G}$, where $\mathcal{G}$ is the set of all possible quantizers (or classifiers). The measure of the accuracy of $g$ will be evaluate thanks to a distortion or risk given by, for some loss function $\ell$:
         \beqn
         \label{risk}
         R(g)=\E_P \ell(g,X)=\int_{\R^d}\ell(g,x)f(x)dx.
          \eeqn
The most investigated example of such a framework is probably cluster analysis, where given some integer $k\geq 2$, we want to build $k$ clusters of the set of observations $X_1,\ldots, X_n$. In this framework, a classifier $g\in\mathcal{G}$ assigns cluster $g(x)\in \{1,\ldots, k\}$ to an observation $x\in\R^d$.
         
         However, in many real-life situations, direct data $X_1,\ldots ,X_n$ are not available and measurement errors occur. Then, we observe only a corrupted sample $Z_i=X_i+\epsilon_i,i=1,\ldots n$ with noisy distribution $\tilde{P}$, where $\epsilon_1,\ldots, \epsilon_n$ are i.i.d. independent of $X_1,\ldots , X_n$ with density $\eta$. The problem of noisy empirical vector quantization or noisy clustering is to represent compactly and efficiently the measure $P$ when a contaminated empirical version $Z_1,\ldots ,Z_n$ is observed. This problem is a particular case of inverse statistical learning (see \cite{loustau12}), and is known to be an inverse problem. To our best knowledge, it has not been yet considered in the literature. This paper tries to fill this gap by giving a theoretical study of this problem. The construction of an algorithm to deal with clustering from a noisy dataset will be the core of a future paper.\\
         
A quiet natural habit in statistical learning is to endow clustering or empirical vector quantization into the general and extensively studied problem of empirical risk minimization (see \citet{vapnik2000},\citet{empimini},\citet{kolt}). This is exactly the guiding thread of this contribution. For this purpose, given a class of classifier or quantizer $\mathcal{G}$ (possibly infinite-dimensional space), let us consider a loss function $\ell:\mathcal{G}\times \R^d$ where $\ell(g,x)$ measures the loss of $g$ at point $x$. In such a framework, given data $X_1,\ldots ,X_n$, it is extremely standard to consider an empirical risk minimizer (ERM) defined as:
\beqn
\label{erm}
\hat{g}_n\in\arg\min_{g\in\mathcal{G}}\frac{1}{n}\sum_{i=1}^n\ell(g,X_i).
\eeqn
Since the pioneer's work of Vapnik, many authors have investigated the statistical performances of (\ref{erm}) in such a generality. We describe below two possible examples that fall into the specific problem of clustering or empirical quantization. \\

\begin{example}[\textsc{The $k$-means clustering problem}]
The finite dimensional clustering problem deals with the construction of a vector $\mathbf{c}= (c_1, \ldots, c_k)\in\mathbb{R}^{dk}$ to represent efficiently with $k\geq 1$ centers a set of observations $X_1,\ldots, X_n\in\R^d$. For this purpose, it is standard to consider the loss function $\gamma:\R^{dk}\times \R^d$ defined as:
$$
\gamma(\mathbf{c},x):=\min_{j=1,\ldots k}\no x-c_j\no^2.
$$
In this case, the empirical risk minimizer is given by $\hat{c}_n=\arg\min\sum_{i=1}^n\min_{j=1,\ldots k}\no X_i-c_j\no^2$ and is known as the popular $k$-means (\citet{pollard81},\citet{pollard82}).
\end{example}
\begin{example}[\textsc{Learning principal curves}] 
Another possible example is to consider quantization with principal curves (see \citet{fisherbiau}). In the definition of \cite{kklz}, a principal curve can be defined as the minimizer of the least-square distortion:
$$
W(g)=\E_P \inf_t \no X-g(t)\no^2,
$$
over a collection of parameterized curves $g:t\mapsto (g_1(t),\ldots,g_d(t))$. Principal curves can be useful in a wide range of statistical learning or data mining problems, such as speech recognition, social sciences or geology (see \cite{fisherbiau} and the references therein). As in (\ref{erm}), we can minimize the empirical least-square distortion $W_n(g)$, namely the distortion integrated with respect to the empirical measure.
\end{example}
%

In this paper, we propose to adopt a comparable strategy in the presence of noisy measurements. Since we observe a corrupted sample $Z_i=X_i+\epsilon_i$, $i=1,\ldots , n$, the empirical risk minimization (\ref{erm}) is not available. However, we can introduce a deconvolution step in the estimation procedure by constructing a kernel deconvolution estimator of the density $f$ of the form:
\beqn
\label{dke}
\hat{f}_\lambda(x)=\frac{1}{n}\sum_{i=1}^n\frac{1}{\lambda}\mathcal{K}_\eta\left(\frac{Z_i-x}{\lambda}\right),
\eeqn
where $\mathcal{K}_\eta$ is a deconvolution kernel and $\lambda=(\lambda_1,\ldots, \lambda_d)\in\R_d^+$ is a regularization parameter (see Section \ref{s:decerm} for details). With a slight abuse of notations, we write in (\ref{dke}), for any $x=(x_1,\ldots, x_d), Z_i=(Z_{1,i},\ldots,Z_{d,i})\in\R^d$:
$$
\frac{1}{\lambda}\mathcal{K}_\eta\left(\frac{Z_i-x}{\lambda}\right)=\frac{1}{\Pi_{i=1}^d\lambda_i}\mathcal{K}_\eta\left(\frac{Z_{1,i}-x_1}{\lambda_1},\ldots,\frac{Z_{d,i}-x_d}{\lambda_d}\right).
$$
Given this estimator, we construct an empirical risk by plugging (\ref{dke}) into the true risk (\ref{risk}) to get
a so-called deconvolution empirical risk minimization. The idea was originated in \cite{pinkfloyds} for discriminant analysis. To fix some notations, in this paper, a solution of this stochastic minimization can be written:
\beqn
\label{dermintro}
\hat{g}_n^\lambda\in\arg\min_{g\in\mathcal{G}}R_n^\lambda(g),\mbox{ where }R_n^\lambda(g)=\frac{1}{n}\sum_{i=1}^n\ell_\lambda(g,Z_i).
\eeqn
Section \ref{s:decerm} is devoted to the detailled construction of the deconvolution empirical risk $R_n^\lambda(\cdot)$, throught the loss $\ell_\lambda(g,\cdot)$. 

The purpose of this work is to study the statistical performances of $\hat g_n^\lambda$ in (\ref{dermintro}) in terms of oracle inequalities. On the one hand, we study the theoretical performances of $\hat{g}_n^\lambda$ thanks to exact oracle inequalities. An exact oracle inequality states that with high probability:
\beqn
\label{exact}
R(\hat{g}_n^\lambda)\leq \inf_{g\in\mathcal{G}}R(g)+r_{n,f,\eta}(\mathcal{G}),
\eeqn
where $r_{n,f,\eta}(\mathcal{G})\longrightarrow 0$ as $n\to\infty$. The residual term $r_{n,f,\eta}(\mathcal{G})$ is called the rate of convergence. It is a function of the complexity of $\mathcal{G}$, the behaviour of the density $f$, and the density of the noise $\eta$. In this paper, the behaviour of $f$  depends on two different assumptions : a margin assumption and a regularity assumption. The margin assumption is related to the difficulty of the problem whereas the regularity assumption will be expressed in terms of anisotropic H\"older spaces.\\
On the other hand, we propose non-exact oracle inequalities, i.e. the existence of a constant $\epsilon>0$, such that with high probability:
\beqn
\label{nonexact}
R(\hat{g}_n^\lambda)\leq (1+\epsilon)\inf_{g\in\mathcal{G}}R(g)+r^\star_{n,f,\eta}(\mathcal{G}).
\eeqn 
The main difference between (\ref{exact}) and (\ref{nonexact}) resides in the residuals which appears in the Right Hand Sides (RHS). As in \cite{lm09}, one of the message of this paper is to highlight the presence of faster rates of convergence (i.e. $r^\star_{n,f,\eta}=o(r_{n,f,\eta})$ as $n\to\infty$) for non-exact oracle inequalities. The cornerstone idea of these results resides in a bias-variance decomposition of the risk $R(\hat{g}_n^\lambda)$ as in \cite{loustau12}. However, in comparison to \cite{loustau12}, this work extend the previous results to unsupervised learning, non-exact oracle inequalities and to an anisotropic class of densities $f$. \\


The paper is organized as follows. In Section \ref{s:decerm}, we present the method and the main assumptions on the density $\eta$ (noise assumption), the kernel in (\ref{dke}) and the density $f$ (regularity and margin assumptions). We state the main theoretical results in Section \ref{s:mainresult}, which consists in exact and non-exact oracle inequalities with fast rates of convergence. It allows to recover recent results in the area of fast rates. These results are applied in Section \ref{application} for the problem of finite dimensional clustering with $k$-means. Section \ref{conclusion} concludes the paper with a discussion whereas Section \ref{proofs}-\ref{appendix} give detailled proofs of the main results.

\section{Deconvolution ERM}
\label{s:decerm}
\subsection{Construction of the estimator}
The deconvolution ERM introduced in this paper is originally due to \cite{pinkfloyds} in discriminant analysis (see also \cite{loustau12} for such a generality in supervised classification). The main idea of the construction is to estimate the true risk (\ref{risk}) thanks to a deconvolution kernel as follows. \\

Let us introduce $\mathcal{K}=\prod_{i=1}^d \mathcal{K}_j:\R^d \to \R$ a $d$-dimensional function defined as the product of $d$ unidimensional function $\mathcal{K}_j$. Besides, $\mathcal{K}$ (and also $\eta$) belongs to $L_2(\R^d)$ and admits a Fourier transform. Then, if we denote by $\lambda=(\lambda_1,\dots,\lambda_d)$ a set of (positive) bandwidths and by $\FF[\cdot]$ the Fourier transform, we define $\mathcal{K}_\eta$ as:
\begin{eqnarray}
\mathcal{K}_{\eta} & : & \R^d \to \R \nonumber \\
& & t \mapsto \mathcal{K}_\eta(t) = \FF^{-1}\left[ \frac{\FF[\mathcal{K}](\cdot)}{\FF[\eta](\cdot/\lambda)}\right](t).
\end{eqnarray}
Given this deconvolution kernel, we construct an empirical risk by plugging (\ref{dke}) into the true risk $R(g)$ to get
a so-called deconvolution empirical risk given by:
\beqn
\label{decer}
R_n^\lambda(g)=\frac{1}{n}\sum_{i=1}^n\ell_\lambda(g,Z_i)\mbox{ where }\ell_\lambda(g,Z_i)=\int_{K} \ell(g,x)\frac{1}{\lambda}\mathcal{K}_\eta\left(\frac{Z_i-x}{\lambda}\right)dx.
\eeqn
Note that for technicalities, we restrict ourselves to a compact set $K\subset\R^d$ and study the risk minimization (\ref{risk}) only in $K$. Consequently, in this paper, we only provide a control of the true risk (\ref{risk}) restricted to $K$, namely the truncated risk:
$$
R_{K}(g)=\int_K\ell(g,x)f(x)dx.
$$
This restriction has been considered in \cite{mammen} (or more recently in \cite{pinkfloyds}). It is important to note that when $f$ has compact support,  we can see coarsely that $R_{K}(g)=R(g)$ for great enough $K$. In the sequel, for simplicity, we write $R(\cdot)$ for the restricted loss defined above. The choice of $K$ is discussed in Section \ref{s:mainresult} and depends on the context. 
\subsection{Assumptions}
For the sake of simplicity, we restrict ourselves to moderately or midly ill-posed inverse problem as follows. We introduce the following noise assumption \textbf{(NA)}:
 \\
 
\noindent
\textbf{(NA)}: There exist $(\beta_1,\dots,\beta_d)'\in \R_+^d$ such that:
$$ \left| \mathcal{F}[\eta](t) \right| \sim \Pi_{i=1}^d|t_i|^{-\beta_i},  \mathrm{as} \ |t_i|\to +\infty,\,\forall i\in\{1,\ldots,d\}.$$
Moreover, we assume that $\mathcal{F}[\eta](t) \not = 0$ for all $t=(t_1,\ldots,t_d)\in \R^d$.\\

Assumption \textbf{(NA)} deals with the asymptotic behaviour of the characteristic function of the noise distribution. These kind of restrictions are standard in deconvolution problems for $d=1$ (see \citet{Fan,meister,butucea}). In this contribution, we only deal with $d$-dimensional mildly ill-posed deconvolution problems, which corresponds to a polynomial decreasing of $\mathcal{F}[\eta]$ in each direction. For the sake of brevity, we do not consider severely ill-posed inverse problems (exponential decreasing)or possible intermediates (e.g. a combination of polynomial and exponential decreasing functions). Recently, \cite{comtelacour} proposes such a study in the context of multivariate deconvolution. In our framework, the rates in these cases could be obtained through the same steps.\\

We also require the following assumptions on the kernel $\mathcal{K}$.
\\

\textbf{(K1)} There exists $S=(S_1,\dots,S_d)\in\R_d^+$, $K_1>0$ such that kernel $\mathcal{K}$ satisfies
$$
\mbox{supp}\mathcal{F}[\mathcal{K}]\subset [-S,S]\mbox{ and }\sup_{t\in\R^d}|\mathcal{F}[\mathcal{K}](t)|\leq K_1,
$$
where $\mbox{supp} \,g=\{x:g(x)\not= 0\}$ and $[-S,S]=\bigotimes_{i=1}^d [-S_i,S_i]$.\\

This assumption is trivially satisfied for different standard kernels, such as the \textit{sinc} kernel. This assumption arises for technicalities in the proofs and can be relaxed using a finer algebra. Moreover, in the sequel, we consider a kernel of order $m$, for a particular $m\in\N^d$.\\

\textbf{K($m$)} The kernel $\mathcal{K}$ is of order $m=(m_1,\ldots, m_d)\in\N^d$, i.e.
\begin{itemize}
\item $\int_{\R^d} \mathcal{K}(x)dx=1$
\item $\int_{\R^d} \mathcal{K}(x)x_j^kdx=0$, $\forall k\leq m_j$, $\forall j\in\{1,\ldots, d\}$.
\item $\int_{\R^d} |\mathcal{K}(x)||x_j|^{m_j}dx<K_2$, $\forall  j\in\{1,\ldots, d\}$.
\end{itemize}

The construction of kernels satisfying \textbf{K($m$)} could be managed as in \citet{booktsybakov}. This property is standard in nonparametric kernel estimation and allows to get satisfying approximations using the following assumption over the regularity of the density $f$.

\begin{definition}
For some $s=(s_1,\ldots, s_d)\in\R_d^+,$ $L>0$, we say that $f$ belongs to the anisotropic H\"older space $\mathcal{H}(s,L)$  if the following holds:
\begin{itemize}
\item the function $f$ admits derivatives with respect to $x_j$ up to order $\lfloor  s_j\rfloor$, where $\lfloor s_j\rfloor$ denotes the largest integer less than $s_j$.
\item $\forall j=1,\ldots ,d$, $\forall x\in\R^d$, $\forall x_j'\in\R$, the following Lipschitz condition holds:
$$
\left|\frac{\partial^{\lfloor s_j\rfloor}}{(\partial x_j)^{\lfloor s_j\rfloor}}f(x_1,\ldots,x_{j-1},x'_j,x_{j+1},\ldots,x_d)-\frac{\partial^{\lfloor s_j\rfloor}}{(\partial x_j)^{\lfloor s_j\rfloor}}f(x)\right|\leq L|x_j'-x_j|^{ s_j-\lfloor  s_j\rfloor}.
$$
\end{itemize}
\end{definition}
If a function $f$ belongs to the anisotropic H\"older space $\mathcal{H}(s,L)$, $f$ has an H\"older regularity $s_j$ in each direction $j=1,\ldots ,d$. As a result, it can be well-approximated pointwise using a $d$-dimensional Taylor formula.
\section{Main results}
\label{s:mainresult}
It is well-known that the behaviour of the rates of convergence $r_{n,f,\eta}(\mathcal{G})$ in (\ref{exact}) or $r^*_{n,f\eta}(\mathcal{G})$ in (\ref{nonexact}) is governed by the size of $\mathcal{G}$. In this paper, the size of the hypothesis space will be quantified in terms of $\epsilon$-entropy with bracketing of the metric space $\left(\{\ell(g),g\in\mathcal{G}\},L_2\right)$ as follows. 
\begin{definition}
Given a metric space $(\mathcal{F},d)$ and a real number $\epsilon>0$, the $\epsilon$-entropy with bracketing of $(\mathcal{F},d)$ is the quantity $\mathcal{H}_B(\mathcal{F},\epsilon,d)$ defined as the logarithm of the minimal integer $N_B(\epsilon)$ such that there exist pairs $(f_j,g_j)\in\mathcal{F}\times\mathcal{F}$, $j=1,\ldots, N_B(\epsilon)$ such that $f_j\leq g_j$, $d(f_j,g_j)\leq \epsilon$, and such that for any $f\in\mathcal{F}$, there exists a pair $(f_j,g_j)$ such that $f_j<f<g_j$.
\end{definition}
This notion of complexity allows to obtain local uniform concentration inequalities (see \cite{vdg} or \cite{wvdv}). Indeed, to reach fast rates of convergence (i.e. faster than $n^{-1/2}$), what really matters is not the total size of the hypothesis space but rather the size of a subclass of $\mathcal{G}$, made of functions with small errors. In this paper, we use an iterative localization principle originally introduced in \cite{koltpachenko} (see also \cite{kolt} for such a generality). More precisely, to state exact oracle inequalities, we consider functions in $\mathcal{G}$ with small excess risk as follows:
$$
\mathcal{G}(\delta)=\{g\in\mathcal{G}:R(g)-\inf_{g\in\mathcal{G}}R(g)\leq \delta\},
$$
whereas to get non-exact oracle inequalities, we consider the following set:
$$
\mathcal{G}'(\delta)=\{g\in\mathcal{G}:R(g)\leq \delta\}.
$$

Originally, \cite{mammen} (see also \cite{tsybakov2004}) formulated an usefull condition to get fast rates of convergence in classification in the exact case. This assumption is known as the margin assumption and has been generalized by \cite{empimini}. coarsely speaking, a margin assumption guarantees a nice relationship between the variance and the expectation of any function of the excess loss class. In this contribution, it appears as follows:
\\

\textbf{Margin Assumption MA($\kappa$)} There exists some $\kappa\geq 1$ such that: 
\beqnn
\,
\forall g\in\mathcal{G}, \no\ell(g,\cdot)-\ell(g^*(g),\cdot)\no_{L_2}^2\leq \kappa_0\left[R(g)-\inf_{g\in\mathcal{G}}R(g)\right]^{1/\kappa},
\eeqnn
for some $\kappa_0>0$ and where $g^*(g)\in\arg\min_{h\in\mathcal{G}}R(h)$ can depend on $g$ when $|\mathcal{G}(0)|\geq 2$.\\

Gathering with a local concentration inequality (see Theorem \ref{concentration} in Section \ref{proofs}) applied to the class $\mathcal{G}(\delta)$, this margin assumption is used in the exact-case to get fast rates. Note that provided that $\ell(g,\cdot)$ is bounded, \textbf{MA($\kappa$)} implies \textbf{MA($\kappa'$)} for any $\kappa'\geq \kappa$. Interestingly, in the framework of finite dimensional clustering with $k$-means, \cite{levrard} proposes to give a sufficient condition to have \textbf{MA($\kappa$)} with $\kappa=1$. This condition is related with the geometry of $f$ with respect to the optimal clusters and gives well-separated classes. It allows to interpret \textbf{MA($\kappa$)} exactly as a margin assumption in clustering (see Section \ref{application}). In the sequel, we call the parameter $\kappa$ in \textbf{MA($\kappa$)} the margin parameter. 

Recently, \cite{lm09} points out that one could wish non-exact oracle inequalities with fast rates under a weaker assumption. The idea is to relax significantly the margin assumption and use the loss class $\{\ell(g),g\in\mathcal{G}\}$ in \textbf{MA($\kappa$)} instead of the excess loss class $\{\ell(g)-\ell(g^*),g\in\mathcal{G}\}$. This framework will be considered at the end of this section for completeness. It leads to non-exact oracle inequalities in the noisy case.
\subsection{Exact Oracle inequalities}
We are now on time to state the main exact oracle inequality.
\begin{theorem}[Exact Oracle Inequality]
\label{mainest}
Suppose \textbf{(NA)}, \textbf{(K1)}, and \textbf{MA($\kappa$)} holds for some margin parameter $\kappa\geq 1$. Suppose $f\in\mathcal{H}(s,L)$ and \textbf{K($m$)} holds with $m=\lfloor s\rfloor$. Suppose there exists $0<\rho<1$, $c>O$ such that for every $\epsilon>0$:
\beqn
\label{complexity}
\mathcal{H}_B(\{\ell(g),g\in\mathcal{G}\},\epsilon,L_2)\leq c\epsilon^{-2\rho}.
\eeqn
Then, for any $t>0$, there exists some $n_0(t)\in\N^*$ such that for any $n\geq n_0(t)$, with probability greater than $1-e^{-t}$, the deconvolution ERM $\hat{g}_n^\lambda$ is such that:
\beqnn 
 R(\hat{g}_n^\lambda)\leq \inf_{g\in\mathcal{G}}R(g) +Cn^{-\tau_d(\kappa,\rho,\beta,s)},
\eeqnn
where $C>0$ is independent of $n$ and $\tau_d(\kappa,\rho,\beta,s)$ is given by:
$$
\tau_d(\kappa,\rho,\beta,s)=\frac{\displaystyle\kappa}{\displaystyle 2\kappa+\rho-1+(2\kappa-1)\sum_{j=1}^d\beta_j/s_j},
$$ 
and $\lambda=(\lambda_1,\ldots,\lambda_d)$ is chosen as:
$$
\lambda_j\approx n^{-\frac{2\kappa-1}{2\kappa s_j}\tau_d(\kappa,\rho,\beta,s)},\forall j=1,\ldots d.
$$
\end{theorem}
The proof of this result is postponed to Section \ref{proofs}. We list some remarks below.
\begin{remark}[Comparison with \cite{kolt} or \cite{mammen}] This result gives the order of the residual term in the exact oracle inequalities. The risk of the estimator $\hat{g}_n^\lambda$ mimics the risk of the oracle, up to a residual term detailled in Theorem \ref{mainest}. The price to pay for the error-in-variables model depends on the asymptotic behaviour of the characteristic function of the noise distribution. If  $\beta=0\in\R^d$ in the noise assumption \textbf{(NA)}, the residual term in Theorem \ref{mainest} satisfies:
$$r_n(\mathcal{G})=O\left(n^{-\frac{\kappa}{2\kappa+\rho-1}}\right).$$
It corresponds to the standard fast rates stated in \cite{kolt} or \cite{mammen} for the direct case.
\end{remark}
\begin{remark}[Comparison with \cite{loustau12}]
In comparison with \cite{loustau12}, these rates deal with an anisotropic behaviour of the density $f$. If $s_j=s$ for any direction, we obtain the same asymptotics as in \cite{loustau12} for supervised classification, namely:
$$r_n(\mathcal{G})=O\left(n^{-\frac{\kappa s }{s(2\kappa+\rho-1)+(2\kappa-1)\sum_{j=1}^d\beta_j}}\right).
$$ 
The result of Theorem \ref{mainest} gives a generalization of \cite{loustau12} to the anisotropic case, in an unsupervised framework. It gives some intuition with respect to the optimality of this result.
\end{remark}
\begin{remark}[The anisotropic case is of practical interest] The result of Theorem \ref{mainest} gives some insights into the noisy quantization problem with an anisotropic density $f$. In this problem, due to the anisotropic behaviour of the density, the choice of the regularization parameters $\lambda_j$, $j=1,\ldots ,d$ depends on $j$. This result is of practical interest since it allows to consider different bandwidth coordinates for the deconvolution ERM. In finite dimensional noisy clustering with $k\geq 2$, this configuration arises when the optimal centers are not uniformly distributed over the support of the density. This case could not be treated at least from theoretical point of view using the previous isotropic approach stated in \cite{loustau12} or \cite{pinkfloyds}.
\end{remark}
\begin{remark}[Fast rates]
The most favorable cases arise when $\rho\to 0$ and $\beta$ is small, whereas at the same time density $f$ has sufficiently high H\"older exponents $s_j$. Indeed, fast rates occur when $\tau_d(\kappa,\rho,\beta,s)\geq 1/2$, or equivalently, $(2\kappa-1)\sum\beta_j/s_j<1-\rho$. If $\rho=0$ and $\kappa=1$ (see the particular case of Section \ref{application}), we have the following condition to get fast rates:
$$
\sum_{j=1}^d\frac{\beta_j}{s_j}<1.
$$
\end{remark}
\begin{remark}[Choice of $\lambda$]
The optimal choice of $\lambda$ in Theorem \ref{mainest} optimizes a bias variance decomposition as in \cite{loustau12}. This choice depends on unknown parameters such as the margin parameter $\kappa$, the H\"older exponents $(s_1,\ldots, s_d)$ of the density $f$ and the degree of illposedness $\beta$. A challenging open problem is to derive adaptive choice of $\lambda$ to lead to the same fast rates of convergence. This could be the purpose of future works.
\end{remark}
\begin{remark}[Comparison with \cite{comtelacour}]
It is also important to note that the optimal choice of the multivariate bandwidth $\lambda$ does not coincide with the optimal choice of the bandwidth in standard nonparametric anisotropic density deconvolution. Indeed, it is stated in \cite{comtelacour} that under the same regularity and ill-posedness assumptions, the optimal choice of the bandwidth $\lambda=(\lambda_1,\ldots,\lambda_d)$ has the following asymptotics:
$$
\lambda_u\approx n^{-\frac{1}{s_u\left(2+\sum_{j=1}^d\frac{2\beta_j+1}{s_j}\right)}}.
$$
The proposed asymptotic optimal calibration of Theorem \ref{mainest} is rather different. It depends explicitely on parameter $\rho$, which measures the complexity of the decision set $\mathcal{G}$, and the margin parameter $\kappa\geq 1$. It shows rather well that our bandwidth selection problem is not equivalent to standard nonparametric estimation problems. It illustrates one more time that our procedure is not a plug-in procedure. 
\end{remark}
\subsection{Non-exact oracle inequalities}
In this section, we also suggest a non-exact version of Theorem \ref{mainest} without the margin assumption \textbf{MA($\kappa$)}. However, to get this result,  we need an additional assumption about the compact $K$ appearing in the empirical risk (\ref{decer}). The assumption has the following form:\\

\textbf{Density assumption DA($c_0$)} There exists a constant $c_0>0$ such that the compact set $K$ in (\ref{decer}) satisfies:
$$
K\subset \{x:f(x)\geq c_0\}.
$$

This assumption is trivially satisfied if $f>0$ in $\R^d$ with a constant $c_0$ depending on the size of $K$. Assumption \textbf{DA($c_0$)} is necessary to get fast rates in the context of non-exact oracle inequalities without the margin assumption \textbf{MA($\kappa$)}. We are now on time to state the following result.
\begin{theorem}[Non-Exact Oracle Inequality]
\label{mainpred}
Suppose \textbf{(NA)}, \textbf{DA($c_0$)} and \textbf{(K1)} holds for some constant $c_0>0$. Suppose $f\in\mathcal{H}(s,L)$ and \textbf{K($m$)} holds with $m=\lfloor s\rfloor$. Suppose there exists $0<\rho<1$, $c>O$ such that for every $\epsilon>0$:
\beqnn
\mathcal{H}_B(\{\ell(g),g\in\mathcal{G}\},\epsilon,L_2)\leq c\epsilon^{-2\rho}.
\eeqnn
Then, for any $t>0$, there exists some $n_0(t)\in\N^*$ such that for any $\epsilon>0$, for any $n\geq n_0(t)$, with probabilty higher than $1-e^{-t}$, $\hat{g}_n^\lambda$ satisfies:
\beqnn 
R(\hat{g})\leq (1+\epsilon)\inf_{g\in\mathcal{G}}R(g) +Cn^{-\tau^*(\rho,\beta,s)},
\eeqnn
where $C>0$ is a constant which depends on $\epsilon,\beta,s,\rho,c_0$ and
$$ 
\tau^*(\rho,\beta,s)=\frac{\displaystyle 1}{\displaystyle1+\rho+\sum_{j=1}^d\beta_j/s_j},
$$
whereas $\lambda=(\lambda_1,\ldots,\lambda_d)$ is chosen as:
$$
\lambda_j\sim n^{-\frac{\tau^*(\rho,\beta,s)}{2s_j}},\forall j=1,\ldots d.
$$
\end{theorem}
\begin{remark}[Same phenomenon as in \cite{lm09}]
The quantity \\$\tau^*(\rho,\beta,s)$ describes the order of the residual term in Theorem \ref{mainpred}. We can see coarsely that $\tau^*(\rho,\beta,s)=\tau(1,\rho,\beta,s)$ where $\tau(1,\rho,\beta,s)$ appears in Theorem \ref{mainest}. As a result, this oracle inequality gives the same asymptotic as the previous result under \textbf{MA($\kappa$)} with  $\kappa=1$, which corresponds to the strong margin assumption. Here, it holds without any margin assumption. The prize to pay is the constant in front of the infimum. This phenomenom has been already pointed out in \cite{lm09} in a supervised framework and in the direct case. Of course, constant $C>0$ in front of the rate depends on $\epsilon>0$ and exploses when $\epsilon$ tends to $ 0$ (see condition (\ref{epsiloncond}) in the proof).
\end{remark}
\begin{remark}[The density assumption] Unfortunately, there is an additional assumption to get Theorem \ref{mainpred} in comparison to Theorem \ref{mainest}, namely the assumption \textbf{DA($c_0$)}. This assumption is specific to the indirect framework where we need to control the variance of the convoluted loss $\ell_\lambda(g,Z)$ with respect to the variance of $\ell(g,X)$. More precisely, we need the following inequality (in dimension $d=1$ for simplicity):
$$
\E_{\tilde P}\ell_\lambda(g,Z)^2\leq \lambda^{-2\beta} \E_P\ell(g,X)^2,\,\forall g\in\mathcal{G}.
$$
This can be done only if we restrict $\ell_\lambda(\cdot)$ to a region where $f>0$. Otherwise, there is no reason to obtain such a control (see Lemma \ref{lipbest} and also the related discussion in \cite{loustau12}).
\end{remark}

\section{Application to finite dimensional noisy clustering}
\label{application}
The aim of this section is to use the general upper bound of Theorem \ref{mainest} in the framework of noisy finite dimensional clustering. To frame the problem of finite dimensional clustering into the general study of this paper, we first introduce the following notation. Given some known integer $k\geq 2$, let us consider $\mathbf{c}= (c_1, \ldots, c_k)\in\mathcal{C}$ the set of possible centers, where $\mathcal{C}\subseteq\R^{dk}$ is compact. The loss function $\gamma:\R^{dk}\times \R^d$ is defined as:
$$
\gamma(\mathbf{c},x)=\min_{j=1,\ldots k}\no x-c_j\no^2,
$$
where $\no\cdot\no$ stands for the standard euclidean norm on $\R^d$. The corresponding true risk or clustering risk is given by $R(\mathbf{c})=\E_P\gamma(\mathbf{c},X)$. In the sequel, we introduce a constant $M\geq 0$ such that $\no X\no_\infty\leq M$. This boundedness assumption ensures $\gamma(\mathbf{c},X)$ to be bounded. The performances of the empirical minimizer $\hat{\mathbf{c}}_n=\arg\min_\mathcal{C}P_n\gamma(\mathbf{c})$ (also called $k$-means clustering algorithm) have been widely studied in the literature. Consistency was shown by \citet{pollard81} when $\E\no X\no^2< \infty$ whereas \citet{llz} or \citet{biau} gives rates of convergence of the form $\mathcal{O}(1/\sqrt{n})$ for the excess clustering risk defined as $R(\hat{\mathbf{c}}_n)-R(c^*)$, where $c^*\in\mathcal{M}$ the set of all possible optimal clusters. More recently, \citet{levrard} proposes fast rates of the form $\mathcal{O}(1/n)$ under Pollard's regularity assumptions. It improves a previous result of \citet{gg}. The main ingredient of the proof is a localization argument in the spirit of \citet{svm}. 

In this section, we study the problem of clustering where we have at our disposal a corrupted sample $Z_i=X_i+\epsilon_i$, $i=1,\ldots ,n$ where the $\epsilon_i$'s are i.i.d. with density $\eta$ satisfying \textbf{(NA)} of Section \ref{s:decerm}. For this purpose, we introduce the following deconvolution empirical risk minimization:
\beqn
\label{noisykmeans}
\arg\min_{\mathbf{c}\in\mathcal{C}}\frac{1}{n}\sum_{i=1}^n\gamma_\lambda(\mathbf{c},Z_i),
\eeqn
where $\gamma_\lambda(\mathbf{c},z)$ is a deconvolution $k$-means loss defined as:
\beqnn
\gamma_\lambda(\mathbf{c},z)=\int_{K}\frac{1}{\lambda}\mathcal{K}_\eta\left(\frac{z-x}{\lambda}\right)\min_{j=1,\ldots k}\|x-c_j\|^2dx.
\eeqnn
The kernel $\mathcal{K}_\eta$ is the deconvolution kernel introduced in Section \ref{s:decerm} with $\lambda=(\lambda_1,\ldots,\lambda_d)\in\R^d_+$ a set of positive bandwidths chosen later on. We investigate the generalization ability of the solution of (\ref{noisykmeans}) in the context of Pollard's regularity assumptions. For this purpose, we will use the following regularity assumptions on the source distribution $P$.

\noindent
\\
\textbf{ Pollard's Regularity Condition (PRC)}: The distribution $P$ satisfies the following two conditions:
                  
                  \begin{enumerate}
                  \item $P$ has a continuous density $f$ with respect to Lebesgue measure on $\mathbb{R}^d$,
                  \item The Hessian matrix of $ \mathbf{c} \longmapsto P \gamma (\mathbf{c},.) $ is positive definite for all optimal vector of clusters $\mathbf{c}^*$.
                  \end{enumerate}
                  It is easy to see that using the compactness of $\mathcal{B}(0,M)$, $\no X\no_\infty\leq M$ and \textbf{(PRC)} ensures that there exists only a finite number of optimal clusters $\c^*\in\mathcal{M}$. This number is denoted as $|\mathcal{M}|$ in the rest of this section.
Moreover, Pollard's conditions can be related to the margin assumption \textbf{MA($\kappa$)} of Section \ref{s:mainresult} thanks to the following lemma due to \citet{gg}.
\begin{lemma}[\citet{gg}]
\label{lemap}
Suppose $\no X\no_\infty\leq M$ and \textbf{(PRC)} holds. Then, for any $\mathbf{c}\in\mathcal{B}(0,M)$:
$$
\no \gamma(\mathbf{c},\cdot)-\gamma(\mathbf{c}^*(\c),\cdot)\no_{L_2}\leq C_1\|\c-\c^*(\c)\|^2\leq C_1C_2\left(R(\c)-R(\c^*(\c))\right),
$$
where $c^*(\c)\in\arg\min_{\c^*}\|\c-\c^*\|$.
\end{lemma}                
Lemma \ref{lemap} ensures a margin assumption \textbf{MA($\kappa$)} with $\kappa=1$ (see Section \ref{s:mainresult}). It is useful to derive fast rates of convergence. Recently, \cite{levrard} has pointed out sufficient conditions to have \textbf{(PRC)} as follows. Denote $\partial V_i$ the boundary of the Voronoi cell $V_i$ associated with $c_i$, for $i=1, \ldots, k$. Then, a sufficient condition to have \textbf{(PRC)} is to control the sup-norm of $f$ on the union of all possible $|\mathcal{M}|$ boundaries $\partial V^{*,m}=\cup_{i=1}^k\partial V^{*,m}_i$, associated with $c^*_m\in\mathcal{M}$ as follows:
$$
\|f_{|\cup_{m=1}^\mathcal{M} \partial V^{*,m}}\|_\infty\leq c(d)M^{d+1}\inf_{m=1,\ldots,|\mathcal{M}|,i=1,\ldots k}P(V_i^{*,m}),
$$
where $c(d)$ is a constant depending on the dimension $d$. As a result, the margin assumption is guaranteed when the source distribution $P$ is well concentrated around its optimal clusters, which is related to well-separated classes. From this point of view, the margin assumption \textbf{MA($\kappa$)} can be related to the margin assumption in binary classification.\\
We are now ready to state the main result of this section.
\begin{theorem}
\label{thm:application}
Assume \textbf{(NA)} holds, $P$ satisfies \textbf{(PRC)} with density $f\in\mathcal{H}(s,L)$ and $\E\no\epsilon\no^2<\infty$. Then, for any $t>0$, for any $n\geq n_0(t)$, denoting by $\hat{\mathbf{c}}^\lambda_n$ a solution of (\ref{noisykmeans}), we have with probability higher than $1-e^{-t}$:
              \beqnn 
R(\hat{\c}^\lambda_n)\leq \inf_{\c\in\mathcal{C}}R(\c) +C\sqrt{\log\log (n )}n^{-\frac{1}{1+\sum_{j=1}^d\beta_j/s_j}},
\eeqnn
where $C>0$ is independent of $n$ and $\lambda=(\lambda_1,\ldots,\lambda_d)$ is chosen as:
$$
\lambda_j\approx n^{-\frac{1}{2s_j(1+\rho+\sum_{j=1}^d\beta_j/s_j)}},\forall j=1,\ldots d.
$$.
                  \end{theorem}
The proof is postponed to Section \ref{proofs}. Here follows some remarks.
\begin{remark}[Fast rates of convergence]
Theorem \ref{thm:application} is a direct application of Theorem \ref{mainest} in Section \ref{s:mainresult}. The order of the residual term in Theorem \ref{thm:application} is comparable to Theorem \ref{mainest}. Due to the finite dimensional hypothesis space $\mathcal{C}\subset\R^{dk}$, we apply the previous study to the case $\rho=0$. It leads to the fast rates $O\left(n^{-\frac{1}{1+\sum_{i=1}^d\beta_j/s_j}}\right)$, up to an extra $\sqrt{\log\log n}$ term. This term is due to the localization principle of the proof, which consists in applying iteratively the concentration inequality of Theorem \ref{concentration}. In the finite dimensional case, when $\rho=0$, we pay an extra $\sqrt{\log\log n}$ term in the rate by solving the fixed point equation. Note that using for instance \cite{levrard}, this term can be avoid. It is out of the scope of the present paper.
\end{remark}
\begin{remark}[Optimality]
Lower bounds of the form $\mathcal{O}(1/\sqrt{n})$ have been stated in the direct case by \citet{lbclustering} for general distribution. An open problem is to derive lower bounds in the context of Theorem \ref{thm:application}. For this purpose, we need to construct configurations where both Pollard's regularity assumption and noise assumption \textbf{(NA)} could be used in a careful way. In this direction, \citet{pinkfloyds} suggests lower bounds in a supervised framework under both margin assumption and \textbf{(NA)}.
\end{remark}
\section{Conclusion}
\label{conclusion}
This paper can be seen as a first attempt into the study of quantization with errors-in-variables. Many problems could be considered in future works, from theoretical or practical point of view.

In the problem of risk minimization with noisy data, we provide oracle inequalities for an empirical risk minimization based on a deconvolution kernel. The risk of the deconvolution ERM mimics the risk of the oracle, up to some residual term, called the rate of convergence. The order of these rates depends on the complexity of the hypothesis space in terms of entropy, the behaviour of the density $f$ and the degree of ill-posedness. From the theoretical point of view, these results extend the previous study of \cite{loustau12} to the unsupervised framework, the non-exact case and to an anisotropic behaviour of the density $f$. These significant extensions could be the core of many applications in unsupervised learning.

As an example, we turn into the problem of clustering with $k$-means. We consider the general approach and introduce a deconvolution kernel estimator of the density $f$ in the distortion. It gives rise to a new stochastic minimization called deconvolution $k$-means. The method gives fast rates of convergence.

Another possible direct application of the result of this paper is to learn principal curves in the presence of noisy observations. In such a problem, the aim is to design a principal curve for an unknown distribution $P$ when we have at our disposal a noisy dataset $Z_i=X_i+\epsilon_i$, $i=1, \ldots ,n$. To the best of our knowledge, this problem has not been considered in the literature. Following the ERM approach of this paper, it is possible to design a new procedure to state rates of convergence in the presence of noisy observations. 

The general deconvolution ERM principle introduced in this paper can be used to design new algorithms to deal with unsupervised statistical learning with noisy observations. As a first step, the construction of a noisy version of the well-known $k$-means is a core of a future work. The construction of a noisy version of the Polygonal Line Algorithm (see \cite{polygonalalgorithm}) could also be investigated, to deal with learning principal curves from indirect observations.
\section{Proofs}
\label{proofs}
The main probabilistic tool for our needs is the following concentration inequality due to Bousquet.
\begin{theorem}[\cite{bousquet}]
\label{concentration}
Let $\mathcal{G}$ a countable class of real-valued measurable functions defined on a measurable space $\mathcal{X}$. Let $X_1,\ldots, X_n$ be $n$ i.i.d. random variables with values in $\mathcal{X}$. Let us consider the random variable:
$$
Z_n(\mathcal{G})=\sup_{g\in\mathcal{G}}\left|\frac{1}{n}\sum_{i=1}^n g(X_i)-\E g(X_1)\right|.
$$
Then, for every $t>0$:
\beqnn
\P(Z_n(\mathcal{G})\geq U_n(\mathcal{G},t))\leq e^{-t},
\eeqnn
where:
$$
U_n(\GG,t)=\E Z_n(\mathcal{G})+\sqrt{\frac{2t}{n}\left[\sigma^2(\GG)+(1+b(\GG))\E Z_n(\GG)\right]}+\frac{t}{3n},
$$
and
$$
\sigma^2(\GG)= \sup_{g\in\GG}\E g(X_1)^2\mbox{ and }b(\GG)=\sup_{g\in\mathcal{G}}\no g\no_{\infty}.
$$
\end{theorem}
The proof of this result uses the so-called entropy method introduced by \cite{ledoux}, and further refined by \cite{massart} or \cite{rio}. The use of a $\psi_1$-version (see for instance \cite{adamczak}) has been considered in \cite{lm09}, to alleviate the boundedness assumption. 

This concentration inequality is at the core of the localization principle presented in  \cite{kolt}, which consists in using Theorem \ref{concentration} to functions in $\mathcal{G}$ with small error. In the following, we extend this localization approach to:
\begin{itemize}
\item the noisy set-up,
\item the non-exact case.
\end{itemize}
For this purpose, we apply Theorem \ref{concentration} to particular classes $\mathcal{G}$, namely excess loss classes for the exact case and loss classes for the non-exact case. These two extensions are proposed in Lemma \ref{lemmaest} and \ref{lemmapred} below. These results are at the core of the general exact and non-exact oracle inequalities of Theorem \ref{mainest} and Theorem \ref{mainpred} in Section \ref{s:mainresult}.
\subsection{Intermediate lemmas}
\subsubsection{Notations}
Let us first introduce the following notations. For any fixed $g\in\mathcal{G}$, we write:
$$
R^\lambda(g)=\int_K\ell(g,x)\E_P\frac{1}{\lambda}\mathcal{K}\left(\frac{X-x}{\lambda}\right)dx\mbox{ and }R_n^\lambda(g)=\frac{1}{n}\sum_{i=1}^n\ell_\lambda(g,Z_i).
$$
As a result, for any fixed $g\in\mathcal{G}$, we have the following equality:
$$
R_n^\lambda(g)-R^\lambda(g)=\frac{1}{n}\sum_{i=1}^n\ell_\lambda(g,Z_i)-\E_{\tilde P}\ell_\lambda(g,Z).
$$
With a slight abuse of notations, we also denote:
$$
(R_n^\lambda-R^\lambda)(g-g')=R_n^\lambda(g)-R^\lambda(g)-R_n^\lambda(g')+R^\lambda(g').
$$
The same notation is used for $R^\lambda(\cdot)$ and $R(\cdot)$ with the quantity $(R-R^\lambda)(g-g')$.

For a function $\psi:\R_+\to\R_+$, the following transformations will be considered:
$$
\breve{\psi}(\delta)=\sup_{\sigma\geq \delta}\frac{\psi(\sigma)}{\sigma}\mbox{ and }\psi^\dagger(\epsilon)=\inf\{\delta>0:\breve{\psi}(\delta)\leq \epsilon\}.
$$
Moreover, we need the following property (see \cite{kolt}):
\beqn
\label{koltprop}
\forall\delta'\leq \delta,\, \psi(\delta)\leq \delta\breve{\psi}(\delta').
\eeqn
We are also interested in the following discretization version of these transformations:
$$
\breve{\psi}_q(\delta)=\sup_{\delta_j\geq \delta}\frac{\psi(\delta_j)}{\delta_j}\mbox{ and }\psi^\dagger_{q}(\epsilon)=\inf\{\delta>0:\breve{\psi}_{q}(\delta)\leq \epsilon\},
$$
where for some $q>1$, $\delta_j=q^{-j}$ for $j\in\mathbb{N}^*$.

Finally, in the sequel, constants $K,C>0$ denote generic constants that may vary from line to line. 
\subsubsection{Exact case}
The proof of Theorem \ref{mainest} uses the following intermediate lemma. 
\begin{lemma}[Exact case]
\label{lemmaest}
Suppose there exists some function $a: \lambda\mapsto a(\lambda)$ and a constant $0<r<1$ such that:
\beqn
\label{biassumption}
\forall g\in\mathcal{G},\,\left|(R-R^\lambda)(g-g^*(g))\right|\leq a(\lambda)+r(R(g)-R(g^{*}(g))),
\eeqn
where $g^*(g)\in\arg\min_h R(h)$ can depend on $g$.\\
Then, for any $q>1$, $\forall \delta\geq \bar{\delta}_{\lambda}(t)$, we have:
\beqnn
\mathbb{P}(R(\hat{g}_n^\lambda)\geq \inf_{g\in\mathcal{G}}R(g)+ \delta)\leq\log_q\left(\frac{1}{\delta}\right) e^{-t},
\eeqnn
where:
$$
\bar{\delta}_\lambda(t)=\max\left(\delta_\lambda(t),\frac{8q}{1-r}a(\lambda)\right),
$$
for $\delta_\lambda(t)=(U_\lambda(\cdot,t))^{\dagger}\left((1-r)/4q\right)$
and where we define, for some constant $K>0$:
\beqnn
U_{\lambda}(\delta,t):=K\left[\E Z_\lambda(\delta)+\sqrt{\frac{t}{n}}\sigma_\lambda(\delta)+\sqrt{\frac{t}{n}\left(1+2b_\lambda(\delta)\right)\E Z_\lambda(\delta)}+\frac{t}{3n}\right],
\eeqnn
where 
\beqnn
Z_\lambda(\delta):=\sup_{g,g'\in\GG(\delta)}\left|(R_n^\lambda-R^\lambda)(g-g')\right|,
\eeqnn
\beqnn
\sigma_\lambda(\delta):=\sup_{g,g'\in\GG (\delta)}\sqrt{\E_{\tilde P}(\ell_\lambda(g,Z)-\ell_\lambda(g',Z))^2},
\eeqnn
\beqnn
b_\lambda(\delta):=\sup_{g\in\mathcal{G}(\delta)}\no\ell_\lambda(g,\cdot)\no_{\infty}.
\eeqnn
\end{lemma}
\begin{proof} The proof follows \cite{kolt} extended to the noisy set-up.

Given $q>1$, we introduce a sequence of positive numbers:
\beqnn
\delta_j=q^{-j},\,\forall j\geq 1.
\eeqnn
Given $n,j\geq 1$, $t>0$ and $\lambda\in\R^d_+$, consider the event:
\beqnn
E_{\lambda,j}(t)=\left\{Z_\lambda(\delta_j)\leq U_{\lambda}(\delta_j,t)\right\}.
\eeqnn
Then, we have, using Theorem \ref{concentration}, for some $K>0$, $\P(E_{\lambda,j}(t)^C)\leq e^{-t}$, $\forall t> 0$.\\
We restrict ourselves to the event $E_{\lambda,j}(t)$. \\
Let $\epsilon<c\delta_{j+1}$ where $c>0$ is chosen later on. Then, consider some $g\in\mathcal{G}(\epsilon)$, where:
$$
\mathcal{G}(\epsilon)=\{g\in\mathcal{G}:R(g)-\inf_{g\in\mathcal{G}}R(g)\leq \epsilon\}.$$
Using assumption (\ref{biassumption}) and the definition of $\hat g:=\hat{g}_n^\lambda$, one has:
\beqnn
R(\hat g)-\inf_{g\in\mathcal{G}}R(g)&\leq& R(\hat g)-R(g)+\epsilon\\
&\leq & (R-R^\lambda)(\hat g-g)+(R^\lambda-R_n^\lambda)(\hat g-g)+\epsilon\\
&\leq & (R^\lambda-R_n^\lambda)(\hat g-g)+2a(\lambda)+r(R(\hat g)-\inf_{g\in\mathcal{G}}R(g))+r(R(g)-\inf_{g\in\mathcal{G}}R(g))+\epsilon\\
\eeqnn
Hence, we have the following assertion:
\beqnn
\delta_{j+1}\leq R(\hat{g})-\inf_{g\in\mathcal{G}}R(g)\leq \delta_j\Rightarrow \delta_{j+1}\leq \frac{1}{1-r}\left((R_n^\lambda-R^\lambda)(g-\hat g)+2a(\lambda)+(1+r)\epsilon\right).
\eeqnn
On the event $E_{\lambda,j}(t)$, it follows that $\forall \delta\leq\delta_j$:
\beqnn
\delta_{j+1}\leq R(\hat{g})-\inf_{g\in\mathcal{G}}R(g)\leq \delta_j\Rightarrow\delta_{j+1}&\leq& \frac{1}{1-r}\left(U_{\lambda}(\delta_j,t)+2a(\lambda)+(1+r)\epsilon\right)\\
&\leq &  \frac{1}{1-r}\left(\delta_jV_{\lambda}(\delta,t)+2a(\lambda)(1+r)\epsilon\right),
\eeqnn
where $V_{\lambda}(\delta,t)=\breve{U}_{\lambda}(\delta,t)$ satisfies property (\ref{koltprop}). We obtain, for any $\delta\leq \delta_j$:
\beqnn
\frac{1}{1-r}V_{\lambda}(\delta,t)\geq \frac{1}{q}-\frac{q^j(2a(\lambda)+(1+r)\epsilon)}{1-r}.
\eeqnn
The assumption $a(\lambda)\leq (1-r)\delta/8q$ and the choice of $c=\frac{1-r}{4(1+r)}$ in the beginning of the proof gives the following lower bound:
$$
V_{\lambda}(\delta,t)>\frac{1-r}{2q}.
$$
It follows from the definition of the $\dagger$-transform that:
\beqnn
\delta< [{U}_\lambda(\cdot,t)]^{\dagger}\left(\frac{1-r}{2q}\right)=\delta_\lambda(t).
\eeqnn
Hence, we have on the event $E_{\lambda,j}(t)$, for any $\delta\leq \delta_j$:
\beqnn
\delta_{j+1}\leq R(\hat{g})-\inf_{g\in\mathcal{G}}R(g)\leq \delta_j\Rightarrow\delta< \delta_n^\lambda(t),
\eeqnn
or equivalently,  
\beqnn
\delta_\lambda(t)\leq \delta\leq \delta_j\Rightarrow \hat{g}\notin\GG(\delta_{j+1},\delta_{j}),
\eeqnn
where $\GG(c,C)=\{g\in\mathcal{G}:c\leq R(g)-\inf_{g\in\mathcal{G}}R(g)\leq C\}$. We eventually obtain:
\beqnn
\bigcap_{\delta_j\geq \delta}E_{\lambda,j}(t)\mbox{ and }\delta\geq \delta_\lambda(t)\Rightarrow R(\hat{g})-\inf_{g\in\mathcal{G}}R(g)\leq \delta.
\eeqnn
This formulation allows us to write by union's bound:
\beqnn
\mathbb{P}(R(\hat{g})\geq \inf_{g\in\mathcal{G}}R(g)+\delta)\leq \sum_{\delta_j\geq \delta} \P(E_{\lambda,j}(t)^C)\leq \log_q\left(\frac{1}{\delta}\right)e^{-t},
\eeqnn
since $\{j:\delta_j\geq \delta\}=\{j:j\leq -\frac{\log \delta}{\log q}\}$.
\end{proof}
\subsubsection{The non-exact case}
The proof of Theorem \ref{mainpred} uses the following version of Lemma \ref{lemmaest}.
\begin{lemma}[Non-exact case]
\label{lemmapred}
Suppose there exists $a^*(\cdot,\cdot):(r,\lambda)\in (0,1)\times \R^+\mapsto a^*(r,\lambda)$ such that for any $(r,\lambda)\in (0,1)\times \R_+$:
\beqn
\label{biassumptionpred}
\forall g\in\mathcal{G}, \left|R(g)-R^\lambda(g)\right|\leq a^*(r,\lambda)+rR(g).
\eeqn
Then, for any $q>1$, $\alpha\in (0,1)$, $u\in (0,1/q)$, $ \delta\geq \bar{\delta}'_\lambda(t)$:
\beqnn
\P(R(\hat{g}_n^{\lambda})\geq \delta)\leq\log\frac{1}{\delta}e^{-t},
\eeqnn
where:
$$
\bar{\delta}'_\lambda(t)=\max\left(\delta'_\lambda(t),\frac{2}{(1-r)\alpha u}a^*(r,\lambda),\frac{1+r}{(1-r)(1-\alpha)u}\inf_{g\in\mathcal{G}}R(g)\right)
$$
for
$$
\delta'_\lambda(t)=\left(U'_\lambda(\cdot,t)\right)^{\dagger}\left(\frac{(1-r)(1-qu)}{2q}\right),
$$
and where we define, for some constant $K>0$:
\beqnn
U'_{\lambda}(\delta,t):=K\left[Z'_{\lambda}(\delta)+\sqrt{\frac{t}{n}}\sigma'_\lambda(\delta)+\sqrt{\frac{t}{n}\left(1+b'_\lambda(\delta)\right)\E Z'_\lambda(\delta)}+\frac{t}{3n}\right],
\eeqnn
where here, we write for $\mathcal{G}'(\delta)=\{g\in\mathcal{G}:R(g)\leq \delta\}$:
\beqnn
Z'_\lambda(\delta):= \sup_{g\in\GG'(\delta)}\left|(R_n^\lambda-R^\lambda)(g)\right|,
\eeqnn
\beqnn
\sigma'_\lambda(\delta):=\sup_{g\in\GG' (\delta)}\sqrt{\E_{\tilde P}(\ell_\lambda(g,Z))^2},
\eeqnn
\beqnn
b'_\lambda(\delta):=\sup_{g\in\mathcal{G}'(\delta)}\no \ell_\lambda(g,\cdot)\no_{\infty}.
\eeqnn
\end{lemma}
\begin{proof}
The proof follows the proof of Lemma \ref{lemmaest} applied to the non-exact case.
Given $q>1$, we introduce a sequence of positive numbers:
\beqnn
\delta_j=q^{-j},\,\forall j\geq 1.
\eeqnn
Given $n,j\geq 1$, $t>0$ and $\lambda\in\R^d_+$, consider the event:
\beqnn
E'_{\lambda,j}(t)=\left\{Z'_\lambda(\delta_j)\leq U'_{\lambda}(\delta_j,t)\right\}.
\eeqnn
Then, we have that, using Theorem \ref{concentration}, $\P(E'_{\lambda,j}(t)^C)\leq e^{-t}$.\\
We restrict ourselves to the event $E'_{\lambda,j}(t)$. \\
Using assumption (\ref{biassumptionpred}), we have, for any $g\in\mathcal{G}$ and any $r\in (0,1)$:
\beqnn
R(\hat{g})\leq\frac{1}{1-r} \left((R^\lambda-R_n^\lambda)(\hat g)+a^*(r,\lambda)+R_n^\lambda(g)\right),
\eeqnn
where we use the definition of $\hat g=\hat{g}_n^\lambda$.
Moreover, note that, using again assumption (\ref{biassumptionpred}):
\beqnn
R_n^\lambda(g)&= & (R_n^\lambda-R^\lambda)(g)+(R^\lambda-R)(g)+R(g)\\
&\leq & (R_n^\lambda-R^\lambda)(g)+a^*(r,\lambda)+(1+r)R(g)
\eeqnn
Then, we have, for $g=g^*\in\arg\min_\mathcal{G} R(g)$:
\beqnn
R(\hat{g})&\leq &\frac{1}{1-r}\left((R_n^\lambda-R^\lambda)(g^*-\hat{g})+2a^*(r,\lambda)+(1+r)\inf_{g\in\mathcal{G}}R(g)\right).
\eeqnn
We hence have on the event $E'_{\lambda,j}(t)$:
\beqnn
\delta_{j+1}\leq R(\hat{g})\leq \delta_j\Rightarrow \delta_{j+1}\leq \frac{1}{1-r}\left(2U'_\lambda(\delta_j,t)+2a^*(r,\lambda)+(1+r)\inf_{g\in\mathcal{G}}R(g)\right),
\eeqnn
since in this case $R(g^*)\leq \delta_j$.
On the event $E'_{\lambda,j}(t)$, it follows that $\forall \delta\leq\delta_j$:
\beqnn
\delta_{j+1}\leq R(\hat{g})\leq \delta_j\Rightarrow\delta_{j+1}&\leq& \frac{1}{1-r}\left(2\delta_jV'_{\lambda}(\delta,t)+2a^*(r,\lambda)+(1+r)\inf_{g\in\mathcal{G}}R(g)\right),
\eeqnn
where $V'_{\lambda}(\delta,t)=\breve{U}'_\lambda(\cdot,t)$ is defined as above.
We obtain, for any $u\in (0,1/q)$:
\beqn
\label{keyineq}
\frac{2}{1-r}V'_{\lambda}(\delta,t)\geq \frac{1}{q}-\frac{q^j}{1-r}\big(2a(\lambda)+(1+r)\inf_{g\in\mathcal{G}}R(g)\big)> \frac{1}{q}-u,
\eeqn
provided that for any $\alpha\in (0,1)$, since $\delta\leq\delta_j$:
$$
a^*(r,\lambda)\leq \alpha\frac{ u (1-r)\delta}{2}\mbox{ and }\inf_{g\in\mathcal{G}}R(g)\leq (1-\alpha)\frac{u(1-r)}{1+r}\delta.
$$
From (\ref{keyineq}), on the event $E_{\lambda,j}(t)$, for any $\frac{1+r}{u(1-\alpha)(1-r)}\inf_{g\in\mathcal{G}}R(g)\vee \frac{2}{(1-r)\alpha u} a^*(r,\lambda)\leq\delta\leq \delta_j$:
\beqnn
\delta_{j+1}\leq R(\hat{g})\leq \delta_j\Rightarrow\delta\leq \delta'_\lambda(t):=[{U}'_\lambda(\cdot,t)]^{\dagger}\left(\frac{1-r}{2q}-\frac{(1-r)u}{2}\right),
\eeqnn
or equivalently, by definition of $\bar{\delta}'_\lambda(t)$:
\beqnn
\bar{\delta}'_\lambda(t) \leq \delta\leq \delta_j\Rightarrow \hat{g}\notin\GG'(\delta_{j+1},\delta_{j}),
\eeqnn
where here $\GG'(c,C)=\{g\in\mathcal{G}:c\leq R(g)\leq C\}$. We eventually obtain:
\beqnn
\bigcap_{\delta_j\geq \delta}E_{\lambda,j}(t)\mbox{ and }\delta\geq \bar{\delta}'_\lambda(t) \Rightarrow R(\hat{g})\leq \delta.
\eeqnn
This formulation allows us to write by union's bound, exactly as in the proof of Lemma \ref{lemmaest}:
\beqn
\label{endineq}
\mathbb{P}(R(\hat{g})\geq \delta)\leq \sum_{\delta_j\geq \delta} \P(E_{\lambda,j}(t)^C)\leq \log_q\left(\frac{1}{\delta}\right)e^{-t},
\eeqn
where $\delta\geq \bar{\delta}'_\lambda(t)$.

\end{proof}

\subsection{Proof of Theorem \ref{mainest} and \ref{mainpred}}
\subsubsection{Proof of Theorem \ref{mainest}}
The proof of Theorem \ref{mainest} is divided into two steps. Using Lemma \ref{lemmaest}, we obtain an exact oracle inequality when $|\mathcal{G}(0)|=1$. For the general case, we will introduce a more sophisticated localization explain in \cite[Section 4]{kolt}. Moreover, we begin the proof in dimension $d=1$ for simplicity. A slightly different algebra is precised at the end of the proof to lead to the general case.\\
\textbf{Case 1: $|\mathcal{G}(0)|=1$.\\}
When $|\mathcal{G}(0)|=1$, it is important to note that \textbf{MA($\kappa$)} holds with a minimizer $g^*\in\mathcal{G}$ which does not depend on $g$. Then, we can write, for any $g,g'\in\mathcal{G}(\delta)$:
$$
\no \ell(g)-l(g')\no_{L_2}\leq \no \ell(g)-l(g^*)\no_{L_2}+\no \ell(g')-l(g^*)\no_{L_2}\leq 2\sqrt{\kappa_0}\delta^{1/2\kappa}.
$$
Gathering with the entropy condition (\ref{complexity}), we obtain:
\beqnn
\E\sup_{g,g'\in\GG(\delta)}\left|(R_n^\lambda-R^\lambda)(g-g')\right|&\leq &\E\sup_{\no \ell(g)-\ell(g')\no_{L_2}\leq 2\sqrt{\kappa_0}\delta^{1/2\kappa}}\left|(R_n^\lambda-R^\lambda)(g-g')\right|\\
&\leq& C\frac{\lambda^{-\beta}}{\sqrt{n}}\delta^{\frac{1-\rho}{2\kappa}},
\eeqnn
where we use in last line Lemma 1 in \citet{loustau12}.
Then, using the notations of Lemma \ref{lemmaest}:
\beqnn
U_\lambda(\delta,t)&=&K\left[\E Z_\lambda(\delta)+\sqrt{\frac{t}{n}}\sigma_\lambda(\delta)+\sqrt{\frac{t}{n}\left(1+2b_\lambda(\delta)\right)\E Z_\lambda(\delta)}+\frac{t}{3n}\right]\\
&\leq & K\left[\frac{\lambda^{-\beta}}{\sqrt{n}}\delta^{\frac{1-\rho}{2\kappa}}+\sqrt{\frac{t}{n}}\sigma_\lambda(\delta)+\sqrt{\frac{t}{n}\left(1+2b_\lambda(\delta)\right)\frac{\lambda^{-\beta}}{\sqrt{n}}\delta^{\frac{1-\rho}{2\kappa}}}+\frac{t}{3n}\right].
\eeqnn
It remains to control the $L^2(\tilde{P})$-diameter $\sigma_\lambda(\delta)$ and the term $b_\lambda(\delta)$ thanks to Lemma \ref{lip}. Using again assumption \textbf{MA($\kappa$)}, and the unicity of the minimizer $g^*$, gathering with the first assertion of Lemma \ref{lip}, we can write:
\beqnn
\sigma_\lambda(\delta)=\sup_{g,g'\in\mathcal{G}(\delta)}\sqrt{\E_{\tilde P}(l_\lambda(g,Z)-l_\lambda(g',Z))^2}\leq C\lambda^{-\beta}\sqrt{\kappa_0}\delta^{\frac{1}{2\kappa}}.
\eeqnn
Now, by the second assertion of Lemma \ref{lip}:
$$
b_\lambda(\delta)=\sup_{g\in\mathcal{G}(\delta)}\no l_\lambda(g,\cdot)\no_{\infty}\leq C\lambda^{-\beta-1/2}.
$$
It follows that:
\beqn
\label{ineq}
U_{\lambda}(\delta,t)\leq K\left[\frac{\lambda^{-\beta}}{\sqrt{n}}\delta^{\frac{1-\rho}{2\kappa}}+\sqrt{t}\frac{\lambda^{-\beta}}{\sqrt{n}}\delta^{\frac{1}{2\kappa}}+\sqrt{\frac{t}{n}\left(1+\lambda^{-\beta-1/2}\right)\frac{\lambda^{-\beta}}{\sqrt{n}}\delta^{\frac{1-\rho}{2\kappa}}}+\frac{t}{3n}\right].
\eeqn
We hence have the following assertion:
$$
t\leq \delta^{-\frac{\rho}{\kappa}}\wedge \sqrt{n}\lambda^{-\beta}\delta^{\frac{1-\rho}{2\kappa}}\Rightarrow U'_{\lambda}(\delta,t)\leq  K \frac{\lambda^{-\beta}}{\sqrt{n}}\delta^{\frac{1-\rho}{2\kappa}}.
$$
From an easy calculation, we hence get in this case:
\beqnn
\delta_{\lambda}(t)\leq K\left(\frac{\lambda^{-\beta}}{\sqrt{n}}\right)^{\frac{2\kappa}{2\kappa+\rho-1}},
\eeqnn
where $K>0$ is a generic constant.
We are now on time to apply Lemma \ref{lemmaest} with:
$$
\delta=K\left(\frac{\lambda^{-\beta}}{\sqrt{n}}\right)^{\frac{2\kappa}{2\kappa+\rho-1}}\mbox{ and }t'=t+\log\log_q n.
$$
In this case, note that for any $t>0$ independent on $n$, the choice of $\lambda$ in Theorem \ref{mainest} warrants that, for any $n\geq n_0(t)$:
$$
t+\log\log_q n\leq \delta^{-\frac{\rho}{\kappa}}\wedge \sqrt{n}\lambda^{-\beta}\delta^{\frac{1-\rho}{2\kappa}}.
$$
Moreover, using Lemma \ref{bias}, we have in dimension $d=1$:
\beqnn
\forall g\in\mathcal{G},\,\left|(R-R^\lambda)(g-g^*)\right|\leq C\lambda^{2s}+\frac{1}{2}(R(g)-R(g^{*})).
\eeqnn
As a result condition (\ref{biassumption}) of Lemma \ref{lemmaest} is satisfied with $r=1/2$ and $a(\lambda)=\lambda^{2s}$. We can also check that for $n$ great enough, the choice of $\lambda$ in Theorem \ref{mainest} guarantees:
$$
\lambda^{2s}\leq K\left(\frac{\lambda^{-\beta}}{\sqrt{n}}\right)^{\frac{2\kappa}{2\kappa+\rho-1}}.
$$
Finally, we get the result since:
$$
\log_q\frac{1}{\delta}e^{-t'}\leq\left(\frac{2\kappa}{2\kappa+\rho-1}\right)\log\left(\frac{\sqrt{n}}{\lambda^{-\beta}}\right)\frac{e^{-t}}{\log_q(n)} \leq e^{-t}.
$$
For the $d$-dimensional case, we have the same algebra by replacing $\lambda^{-\beta}$ by $\Pi_{j=1}^d\lambda_j^{-\beta_j}$ in the previous calculus and $\lambda^{2s}$ by $\sum_{j=1}^d\lambda_j^{2s_j}$ thanks to Lemma \ref{bias}. The choice of $\lambda_j$, for $j=1,\ldots, d$ in Theorem \ref{mainest} allows to conclude.
\\
\textbf{Case 2: $|\mathcal{G}(0)|\geq 2$.\\}
When the infimum is not unique, the diameter $\sigma^2_\lambda(\delta)$ does not necessary tend to zero when $\delta\to 0$. We hence introduce the more sophisticated geometric parameter:
\beqnn
r(\sigma,\delta)=\sup_{g\in\mathcal{G}(\delta)}\inf_{g'\in\mathcal{G}(\sigma)}\sqrt{\E_{\tilde{P}}(\ell_\lambda(g,Z)-\ell_\lambda(g',Z))^2},\mbox{ for }0<\sigma\leq \delta.
\eeqnn
It is clear that $r(\sigma,\delta)\leq \sqrt{\sigma^2_\lambda(\delta)}$ and for $\delta\to 0$, we have $r(\sigma,\delta)\to 0$. The idea of the proof is to use a modified version of Lemma \ref{lemmaest} following \cite[Theorem 4]{kolt}. More precisely, we have to apply the concentration inequality of Theorem \ref{concentration} to the random variable:
$$
W_\lambda(\delta)=\sup_{g\in\mathcal{G}(\sigma)}\sup_{g'\in\mathcal{G}(\delta):\sqrt{\E_{\tilde P}(\ell_\lambda(g,Z)-\ell_\lambda(g',Z))^2}\leq r(\sigma,\delta)+\epsilon}\left|(R_n^\lambda-R^\lambda)(g-g')\right|.
$$
This localization guarantees the upper bounds of Theorem \ref{mainest} when $|\mathcal{G}(0)|\geq 2$. However, to this end, we have to check (for $d=1$ for simplicity):
\beqn
\label{newmodulus1}
\lim_{\epsilon\to 0}\E\sup_{g\in\mathcal{G}(\sigma)}\sup_{g'\in\mathcal{G}(\delta):\sqrt{\E_{\tilde P}(\ell_\lambda(g,Z)-\ell_\lambda(g',Z))^2}\leq r(\sigma,\delta)+\epsilon}\left|(R_n^\lambda-R^\lambda)(g-g')\right|\leq C \frac{\lambda^{-\beta}}{\sqrt{n}}\delta^{1/2\kappa},
\eeqn
and for $0<\sigma\leq\delta$:
\beqn
\label{newdiam1}
r(\sigma,\delta)\leq C\lambda^{-\beta}\delta^{1/2\kappa}.
\eeqn
Using \textbf{MA($\kappa$)} and Lemma 1 in \cite{loustau12}, it is clear that (\ref{newmodulus1}) holds since:
\beqnn
&&\E\sup_{g\in\mathcal{G}(\sigma)}\sup_{g'\in\mathcal{G}(\delta):\sqrt{\E_{\tilde P}(\ell_\lambda(g,Z)-\ell_\lambda(g',Z))^2}\leq r(\sigma,\delta)+\epsilon}\left|(R_n^\lambda-R^\lambda)(g-g')\right|\\
&&\leq \E\sup_{g\in\mathcal{G}(\sigma),g^*\in\mathcal{G}(0)}\left|(R_n^\lambda-R^\lambda)(g-g^*)\right|+\E\sup_{g'\in\mathcal{G}(\delta)}\left|(R_n^\lambda-R^\lambda)(g'-g^*(g'))\right|\\
&&\leq 2\mathbb{E} \sup_{(g,g^*)\in\mathcal{G}(\delta)\times\mathcal{G}(0)} \left|(R_n^\lambda-R^\lambda)(g^* - g)\right| \\
&&\leq  C\frac{\lambda^{-\beta}}{\sqrt{n}}\delta^{1/2\kappa}.
\eeqnn
To check (\ref{newdiam1}), note that with \textbf{MA($\kappa$)} and the first assertion of Lemma \ref{lip}, we have $\forall g\in\mathcal{G}(\delta),g'\in\mathcal{G}(\sigma)$:
\beqnn
\sqrt{\E_{\tilde{P}}(\ell_\lambda(g,Z)-\ell_\lambda(g',Z))^2}&\leq& C\lambda^{-\beta}\no\ell(g)-\ell(g')\no_{L_2}\\
&\leq&C\lambda^{-\beta}\delta^{1/2\kappa}+C\lambda^{-\beta}\no\ell (g^*(g))-\ell (g^*(g'))\no_{L_2},
\eeqnn
for $0<\sigma\leq\delta$. Taking the infimum with respect to $g'\in\mathcal{G}(\sigma)$, we get: 
$$\no\ell (g^*(g))-\ell (g^*(g'))\no_{L_2}=0.$$
\subsubsection{Proof of Theorem \ref{mainpred}}
The main ingredient of the proof is Lemma \ref{lemmapred}. We want to find a convenient bound for the term (see the notations of Lemma \ref{lemmapred}):
\beqnn
U'_{\lambda}(\delta,t)=K\left[Z'_\lambda(\delta)+\sqrt{\frac{t}{n}}\sigma'_\lambda(\delta)+\sqrt{\frac{t}{n}\left(1+b'_\lambda(\delta)\right)\E Z'_\lambda(\delta)}+\frac{t}{3n}\right].
\eeqnn
First note that since $\ell(g,\cdot)$ is bounded, we have the crude bound $\E_{P}\ell(g,X)^2\leq MR(g)$, where $M=\no \ell(g,\cdot)\no_\infty$. Hence, we have, using the entropy condition:
\beqnn
\E Z'_\lambda(\delta)&=&\E\sup_{g\in\GG'(\delta)}\left|(R_n^\lambda-R^\lambda)(g)\right|\\
&\leq &\E\sup_{\no \ell(g)\no_{L_2(P)}\leq \sqrt{M}\delta^{1/2}}\left|(R_n^\lambda-R^\lambda)(g)\right|\\
&\leq& C\frac{\lambda^{-\beta}}{\sqrt{n}}\delta^{\frac{1-\rho}{2}},
\eeqnn
where we use in last line Lemma 1 in \cite{loustau12}.\\ We obtain:
\beqnn
U'_{\lambda}(\delta,t)&\leq & K\left[\frac{\lambda^{-\beta}}{\sqrt{n}}\delta^{\frac{1-\rho}{2}}+\sqrt{\frac{t}{n}}\sigma'_\lambda(\delta)+\sqrt{\frac{t}{n}\left(1+b'_\lambda(\delta)\right)\frac{\lambda^{-\beta}}{\sqrt{n}}\delta^{\frac{1-\rho}{2}}}+\frac{t}{3n}\right].
\eeqnn
Now, from Lemma \ref{lipbest}, we have the following control of $\sigma'_\lambda(\delta)$:
\beqnn
\sigma'_\lambda(\delta)=\sup_{g\in\mathcal{G}'(\delta)}\sqrt{E_{\tilde{P}}\ell_\lambda(g)^2}\leq C\lambda^{-\beta}\sqrt{\E\ell(g,X)^2}\leq C\lambda^{-\beta}\sqrt{\delta},
\eeqnn
where $C>0$ is a generic constant and where we use in the last inequality the boundedness assumption of $\ell(g,\cdot)$.
Now by the second assertion of Lemma \ref{lip}:
$$
b'_\lambda(\delta)=\sup_{g\in\mathcal{G}(\delta)}\no l_\lambda(g,\cdot)\no_{\infty}\leq C\lambda^{-\beta-1/2}.
$$
It follows that:
\beqn
U'_{\lambda}(\delta,t)\leq K\left[\frac{\lambda^{-\beta}}{\sqrt{n}}\delta^{\frac{1-\rho}{2}}+\sqrt{\frac{t}{n}}\lambda^{-\beta}\delta^{\frac{1}{2}}+\sqrt{\frac{t}{n}}\frac{\lambda^{-\beta}}{\sqrt{n}}+\sqrt{\frac{t}{n}\left(1+\lambda^{-\beta-1/2}\right)\frac{\lambda^{-\beta}}{\sqrt{n}}\delta^{\frac{1-\rho}{2}}}+\frac{t}{3n}\right].
\eeqn
We hence have in this case the following assertion:
$$
t\leq \delta^{-2\rho}\wedge n\delta^{-\rho}\wedge \sqrt{n\lambda}\delta^{\frac {1-\rho}{2}}\Rightarrow U'_{\lambda}(\delta,t)\leq  K \frac{\lambda^{-\beta}}{\sqrt{n}}\delta^{\frac{1-\rho}{2}}.
$$
From an easy calculation, we hence get with the notations of Lemma \ref{lemmapred}:
\beqn
\label{condition1}
\delta'_{\lambda}(t)\leq K\left(\frac{\lambda^{-\beta}}{\sqrt{n}}\right)^{\frac{2}{1+\rho}},
\eeqn
where $K>0$ is a generic constant. Let us consider, for any $\epsilon>0$:
$$
\delta=\frac{K\vee 2C}{\alpha_\epsilon u_\epsilon(1-r_\epsilon)r_\epsilon}\left(\frac{\lambda^{-\beta}}{\sqrt{n}}\right)^{\frac{2}{1+\rho}}+(1+\epsilon)\inf_{g\in\mathcal{G}}R(g),
$$
where $(r_\epsilon,\alpha_\epsilon,u_\epsilon)\in (0,1)^2\times (0,1/q)$ are chosen later on as a function of $\epsilon>0$.
Using Lemma \ref{biasbest}, we have in dimension $d=1$, for any $r\in (0,1)$:
\beqnn
\forall g\in\mathcal{G},\,\left|(R-R^\lambda)(g)\right|\leq \frac{C}{r}\lambda^{2s}+rR(g).
\eeqnn
As a result, condition (\ref{biassumptionpred}) of Lemma \ref{lemmapred} is satisfied with $a^*(r,\lambda)=C\lambda^{2s}/r$. The choice of $\lambda$ in Theorem \ref{mainpred} warrants that:
\beqn
\label{condition2}
\lambda^{2s}\leq \left(\frac{\lambda^{-\beta}}{\sqrt{n}}\right)^{\frac{2}{1+\rho}}.
\eeqn
Moreover, for any $\epsilon>0$, we can find a triplet $(r_\epsilon,\alpha_\epsilon,u_\epsilon)\in (0,1)^2\times (0,1/q)$ such that:
\beqn
\label{epsiloncond}
1+\epsilon\geq \frac{1+r_\epsilon}{(1-r_\epsilon)u_\epsilon(1-\alpha_\epsilon)}.
\eeqn
Inequalities (\ref{condition1}), (\ref{condition2}) and  (\ref{epsiloncond}) give us:
$$
\delta\geq \max\left(\delta'_\lambda(t),\frac{1+r_\epsilon}{(1-r_\epsilon)u_\epsilon(1-\alpha_\epsilon)}\inf_{g\in\mathcal{G}}R(g),\frac{2}{(1-r_\epsilon)\alpha_\epsilon u_\epsilon}a^*(r_\epsilon,\lambda)\right).
$$
Finally, we can apply Lemma \ref{lemmapred} with the triplet $(r_\epsilon,\alpha_\epsilon,u_\epsilon)$, $t'=t+\log\log_q n$ and get the result since:
$$
\log_q\frac{1}{\delta}e^{-t'}\leq \frac{2}{1+\rho}\log \left(\frac{\sqrt{n}}{\lambda^{-\beta}}\right)\frac{e^{-t}}{\log_q n}\leq e^{-t}.
$$
\subsection{Proof of Theorem \ref{thm:application}}
The proof of Theorem \ref{thm:application} uses a slightly different version of Theorem \ref{mainest}. First of all, an inspection of the proof of Theorem \ref{mainest} shows that condition (\ref{complexity}) in Theorem \ref{mainest} can be replaced by the following control of the local complexity of the noisy empirical process:
\beqn
\label{localcomp}
\E\sup_{g,g'\in\mathcal{G}(\delta)}\left|(R_n^\lambda-R^\lambda)(g-g')\right|\leq C\frac{\lambda^{-\beta}}{\sqrt{n}}\delta^{\frac{1-\rho}{2\kappa}}.
\eeqn
Hence, using Lemma \ref{complexitync} in the Appendix, gathering with condition \textbf{(PRC)}, we can have (\ref{localcomp}) with $\rho=0$.\\
However, the case $\rho=0$ is not treated in Theorem \ref{mainest} where $\rho\in (0,1)$. From (\ref{localcomp}), and using the notations of Lemma \ref{lemmaest}, (\ref{ineq}) in the proof of Theorem \ref{mainest} becomes:
\beqnn
U_{\lambda}(\delta,t)\leq K\left[\frac{\lambda^{-\beta}}{\sqrt{n}}\delta^{\frac{1}{2}}+\sqrt{t}\frac{\lambda^{-\beta}}{\sqrt{n}}\delta^{\frac{1}{2}}+\sqrt{\frac{t}{n}\left(1+\lambda^{-\beta-1/2}\right)\frac{\lambda^{-\beta}}{\sqrt{n}}\delta^{\frac{1}{2}}}+\frac{t}{3n}\right].
\eeqnn
We hence have the following assertion:
$$
t\leq  \sqrt{n}\lambda^{-\beta}\delta^{\frac{1}{2}}\Rightarrow U_{\lambda}(\delta,t)\leq  K\left(1+\sqrt{t}\right) \frac{\lambda^{-\beta}}{\sqrt{n}}\delta^{\frac{1}{2}}.
$$
Using the same algebra as above, we can use Lemma \ref{lemmaest} with:
$$
\delta=K\left(1+\sqrt{t'}\right)\left(\frac{\lambda^{-\beta}}{\sqrt{n}}\right)^{\frac{2}{1+\rho}}\mbox{ and }t'=t+\log\log_q n.
$$
In this case, note that the choice of $t'=t+\log\log_q n$ gives rise to the following asymptotic:
$$\delta\approx\sqrt{\log\log n}\frac{\lambda^{-\beta}}{\sqrt{n}}\delta^{\frac{1}{2}},$$
and leads to an extra $\sqrt{\log\log n}$ term in the rates of convergence.%
\section{Appendix}
\label{appendix}
\subsection{Technical lemmas for the exact case}
\begin{lemma}
\label{lip}
Suppose \textbf{(NA)} holds, and $\mathcal{K}$ satisfies assumption \textbf{(K1)}. Suppose $\no f*\eta\no_\infty\leq \tilde c_\infty$ and $\sup_{g\in\mathcal{G}}\no \ell(g,\cdot)\no_{L_2(K)}<\infty$. Then, the two following assertions hold:
\begin{description}
\item[(i)] $\ell(g)\mapsto \ell_\lambda(g)$ is Lipschitz with respect to $\lambda$:
\beqnn
\forall g,g'\in\GG,\,\Arrowvert \ell_\lambda(g,\cdot)-\ell_\lambda(g',\cdot)\Arrowvert_{L_2(\tilde{P})}\leq  C_1\Pi_{i=1}^d\lambda_i^{-\beta_i}\Arrowvert \ell(g,\cdot)-\ell(g',\cdot)\Arrowvert_{L_2},
\eeqnn
where $C>0$ is a generic constant which depends on $\tilde c_\infty$ and constants in \textbf{(K1)}.
\item[(ii)] $\{\ell_\lambda(g),g\in\GG\}$ is uniformly bounded:
$$
\sup_{g\in\GG}\| \ell_\lambda(g,\cdot)\|_\infty\leq  C_2\Pi_{i=1}^d\lambda_i^{-(\beta_i+1/2)},
$$
where $C_2>0$ is a generic constant which depends on constants in \textbf{(K1)}.
\end{description}
\end{lemma}
\begin{proof}
Using Plancherel and the boundedness assumption over $f*\eta$, we have: 
\beqnn
\E_{\tilde P}(\ell_\lambda(g,Z)-\ell_\lambda(g',Z))^2
& = & \int \left[ \frac{1}{\lambda}\mathcal{K}_\eta(\frac{\cdot}{\lambda})*(\ind_K\times(\ell(g,\cdot)-\ell(g',\cdot))(z) \right]^2 f*\eta(z)dz \\
& \leq &C \int \frac{1}{\lambda^2}|\mathcal{F}[\mathcal{K}_\eta(\frac{\cdot}{\lambda})](t)|^2|\mathcal{F}[\ind_K\times (\ell(g,\cdot)-\ell(g',\cdot))](t)|^2 dt \\
& \leq& C\lambda^{-2\beta} \no \ell(g)-\ell(g')\no_{L_2}^2,
\eeqnn 
where we use in last line the following inequalities:
\beqnn
\frac{1}{\lambda^2} \left| \mathcal{F}[\mathcal{K}_\eta(./\lambda)](s) \right|^2 =  \left| \mathcal{F}[\mathcal{K}_\eta](s\lambda) \right|^2  \leq C\sup_{t\in \mathbb{R}} \left| \frac{\mathcal{F}[\mathcal{K}](t\lambda)}{\mathcal{F}[\eta](t)} \right|^2 \leq C\sup_{t\in [-\frac{L}{\lambda},\frac{L}{\lambda}]}  \left| \frac{1}{\mathcal{F}[\eta](t)} \right|^2 \leq C  \lambda^{-2\beta},
\eeqnn
provided that \textbf{(K1)} holds. 

By the same way, the second assertion holds since if $\ell(g,\cdot)\in L^2(K)$:
\begin{eqnarray*}
| \ell_\lambda(g,z)|
& \leq & \int_{K} \left| \frac{1}{\lambda}\mathcal{K}_\eta \left(\frac{z-x}{\lambda} \right)\ell(g,x)\right| dx \\
& \leq& C\sqrt{\int_{K} \left|\frac{1}{\lambda} \mathcal{K}_\eta \left(\frac{z-x}{\lambda} \right)\right|^2 dx} \\
& \leq &C \lambda^{-\beta-1/2}.
\end{eqnarray*}
A straightforward generalization leads to the $d$-dimensional case.
\end{proof}

\begin{lemma}
\label{bias}
Suppose $f$ belongs to the anisotropic H\"older spaces $\mathcal{H}(s,L)$ with $s=(s_1,\ldots, s_d)$. Let $\mathcal{K}$ a kernel satisfying assumption \textbf{K($m$)} with $m=\lfloor s \rfloor\in\N^d$. Suppose \textbf{MA($\kappa$)} holds with parameter $\kappa\geq 1$. Then, we have:
\beqnn
\forall g\in\mathcal{G},\,\left|(R-R^\lambda)(g-g^*(g))\right|\leq C\sum_{j=1}^d\lambda_j^{2\kappa s_j/(2\kappa-1)}+\frac{1}{2\kappa}(R(g)-\inf_{g\in\mathcal{G}}R(g)),
\eeqnn
where $C>O$ is a generic constant.
\end{lemma}
\begin{proof}
Note that we can write:
\beqnn
(R^\lambda-R)(g-g^*)
&=&\int_{K}(\ell(g,x)-\ell(g^*,x))\left(\E\hat{f}_\lambda(x)-f(x)\right)dx,
\eeqnn
where we omit the notation $g^*=g^*(g)$ for simplicity.
The first part of the proof uses Proposition 1 stated in \cite{comtelacour}.
\begin{proposition}[\cite{comtelacour}]
Let $B_0(\lambda)=\sup_{x_0\in\R^d}|f(x_0)-\E\hat{f}_\lambda(x_0)|$. Then, if $f$ belongs to the anisotropic H\"older space $\mathcal{H}(s,L)$, and $\mathcal{K}$ is a kernel of order $\lfloor s\rfloor$, we have:
$$
B_0(\lambda)\leq C \sum_{j=1}^d\lambda_j^{s_j},
$$
where $C>0$ denotes some generic constant.
\end{proposition}
The rest of the proof uses the margin assumption \textbf{MA($\kappa$)} as follows:
\beqnn
\left|(R^\lambda-R)(g-g^*)\right|
&\leq&C \sum_{j=1}^d\lambda_j^{s_j}\int_{K}|\ell(g,x)-\ell(g^*,x)|dx.\\
&\leq&C \sum_{j=1}^d\lambda_j^{s_j}\sqrt{\int_{K}|\ell(g,x)-\ell(g^*,x)|^2dx}\\
&\leq &C\sum_{j=1}^d\lambda_j^{s_j} \left(R(g)-R(g^*)\right)^{\frac{1}{2\kappa}}\\
&\leq & C\sum_{j=1}^d\lambda_j^{2\kappa s_j/(2\kappa-1)}+\frac{1}{2\kappa}(R(g)-\inf_{g\in\mathcal{G}}R(g)),
\eeqnn
where we use in last line Young's inequality:
$$
xy^r\leq ry+x^{1/1-r},\forall r<1,
$$
with $r=\frac{1}{2\kappa}$.
\end{proof}
\subsection{Technical lemmas for the non-exact case}

\begin{lemma}
\label{lipbest}
Suppose \textbf{(NA)} and \textbf{DA($c_0$)} holds, and $\mathcal{K}$ satisfies assumption \textbf{(K1)}. Suppose $\no f*\eta\no_\infty\leq \tilde c_\infty$ and $\sup_{g\in\mathcal{G}}\no \ell(g,\cdot)\no_{L_2(K)}<\infty$. Then, we have:
\beqnn
\forall g\in\GG,\,\sqrt{\E_{\tilde{P}}\ell_\lambda(g,Z)^2}\leq  C'_1\Pi_{i=1}^d\lambda_i^{-\beta_i}\sqrt{\E_P\ell(g,X)^2},
\eeqnn
where $C_1'>0$ is a generic constant which depends on $c_0$, $\tilde c_\infty$ and constants in \textbf{(K1)}.
\end{lemma}
\begin{proof}
Using Plancherel and the boundedness assumption over $f*\eta$, we have as above: 
\beqnn
\E_{\tilde P}\ell_\lambda(g,Z)^2 & = & \int \left[ \frac{1}{\lambda}\mathcal{K}_\eta(\frac{\cdot}{\lambda})*\ind_K\times\ell(g,\cdot)(z) \right]^2 f*\eta(z)dz \\
& \leq & C\lambda^{-2\beta}\int_K |\ell(g,z)|^2 dz\\
&\leq&C\frac{\lambda^{-2\beta}}{c_0}\int_K |\ell(g,z)|^2 f(z)dz\\
&\leq&C\lambda^{-2\beta}P\ell(g,X)^2,
\eeqnn
where we use in the third line assumption \textbf{DA($c_0$)}.
\end{proof}
\begin{lemma}
\label{biasbest}
Suppose $f$ belongs to the anisotropic H\"older spaces $\mathcal{H}(s,L)$ with $s=(s_1,\ldots, s_d)$. Let $\mathcal{K}$ a kernel satisfying assumption \textbf{K($m$)} with $m=\lfloor s \rfloor$. Then, we have, for any $r>0$:
\beqnn
\forall g\in\mathcal{G},\,\left|R(g)-R^\lambda(g)\right|\leq \frac{C}{r}\sum_{j=1}^d\lambda_j^{2s_j}+rR(g),
\eeqnn
where $C>O$ is a generic constant which does not depend on $r>0$.
\end{lemma}
\begin{proof}
We follow the first part of the proof of Lemma \ref{bias} to get:
\beqnn
\left|R^\lambda(g)-R(g)\right|
&\leq&C\sum_{j=1}^d\lambda_j^{s_j}\int_{K}|\ell(g,x)|dx.
\eeqnn
Now using \textbf{DA($c_0$)}, we have, for any $r>0$:
\beqnn
\left|R^\lambda(g)-R(g)\right|
&\leq&C\sum_{j=1}^d\lambda_j^{s_j}\sqrt{\int_{K}|\ell(g,x)|^2dx}\\
&\leq&\frac{C\sum_{j=1}^d\lambda_j^{s_j}}{\sqrt{c_0}}\sqrt{\E_P\ell(g,X)^2}\\
&\leq &C\sum_{j=1}^d\lambda_j^{s_j}\left(R(g)\right)^{\frac{1}{2}}\\
&= &\frac{C}{\sqrt{2r}}\sum_{j=1}^d\lambda_j^{s_j}\left(2rR(g)\right)^{\frac{1}{2}}\\
&\leq & \frac{C}{2r}\sum_{j=1}^d\lambda_j^{2s_j}+rR(g),
\eeqnn
where we use in last line Young's inequality:
$$ 
xy^a\leq ay+x^{1/1-a},\forall a<1,
$$
with $a=\frac{1}{2}$.
\end{proof}
\subsection{Technical lemma for Theorem \ref{thm:application}}

\begin{lemma}
\label{complexitync}
Suppose \textbf{(PRC)}, \textbf{(NA)} and the kernel assumption \textbf{(K1)} are satisfied and $\no X\no_{\infty}\leq M$. Suppose $\E\|\epsilon\|^2<\infty$. Then: 
$$
\mathbb{E} \sup_{(\c,\c^*)\in\mathcal{C}\times\mathcal{M}, \|\mathbf{c}-\mathbf{c}^*\|^2 \leq \delta} {\left|(R_n^\lambda-R^\lambda)(\mathbf{c}^*- \mathbf{c})\right|} \leq C \Pi_{i=1}^d\lambda_i^{-\beta_i}\frac{\sqrt{\delta}}{\sqrt{n}},
$$
where $C>0$ is a positive constant.
\end{lemma} 
\begin{proof}
The proof follows \citet{levrard} applied to the noisy setting. First note that in the sequel, we need to introduce the following notation:
$$
(\tilde{P}_n-\tilde{P})(\gamma_\lambda(\c,Z)-\gamma_\lambda(\c',Z):=\frac{1}{n}\sum_{i=1}^n\left[\gamma_\lambda(\c,Z_i)-\gamma_\lambda(\c',Z_i)\right]-\E_{\tilde P}\left[\gamma_\lambda(\c,Z)-\gamma_\lambda(\c',Z)\right].
$$
By smoothness assumptions over $\mathbf{c}\mapsto \min \Arrowvert x-c_j\Arrowvert$,
for any $\mathbf{c} \in \mathbb{R}^{dk}$ and $\mathbf{c}^* \in \mathcal{M}$,               we have:
               \[
               \gamma_\lambda(\mathbf{c},z)-\gamma_\lambda(\mathbf{c}^*,z)= \left\langle \mathbf{c}-\mathbf{c}^*,\nabla_{\c}\gamma_\lambda(\mathbf{c}^*,z) \right\rangle + \|\mathbf{c}-\mathbf{c}^*\|R_\lambda(\mathbf{c}^*,\mathbf{c}-\mathbf{c}^*,z),
               \]
               where, with \cite{pollard82} we have:
$$
 \nabla_{\c}\gamma_\lambda(\mathbf{c}^*,z) = -2\left( \int\frac{1}{\lambda}\mathcal{K}_\eta\left(\frac{z-x}{\lambda}\right)(x-c^*_1)\mathbf{1}_{V^*_1}(x)dx,...,\int\frac{1}{\lambda}\mathcal{K}_\eta\left(\frac{z-x}{\lambda}\right)(x-c^*_k)\mathbf{1}_{V^*_k}(x)dx \right) $$
and $R_\lambda(          \mathbf{c}^*,\mathbf{c}-\mathbf{c}^*,z)$ satisfies:
$$
|R_\lambda(\mathbf{c}^*,\mathbf{c}-\mathbf{c}^*,z)|\leq\| \c - \c^* \|^{-1} \left(\left|\left\langle\mathbf{c}-\mathbf{c}^*,\nabla_{\c}\gamma_\lambda(\mathbf{c}^*,z)\right\rangle\right|+\max_{j=1,\ldots k}(|\|z-\mathbf{c}_j\|-\|x-\mathbf{c}^*_j\|\right).
$$
Splitting the expectation in two parts, we obtain:
\beqn
\label{dec}
&&\hspace{-0.9cm}\mathbb{E}  \sup_{\c^* \in \mathcal{M}, \|\mathbf{c}-\mathbf{c}^*\|^2 \leq \delta} {|\tilde{P}_n-\tilde{P}|(\gamma_\lambda(\mathbf{c}^*,.) - \gamma_\lambda(\mathbf{c},.))} \leq 
               \mathbb{E}   \sup_{\c^* \in \mathcal{M}, \|\mathbf{c}-\mathbf{c}^*\|^2 \leq \delta} { |\tilde{P}_n-\tilde{P}| \left\langle \mathbf{c}^*-\mathbf{c}, \nabla_{\c}\gamma_\lambda(\mathbf{c}^*,.) \right\rangle  }   \nonumber\\
               &+& \sqrt{\delta}  \mathbb{E} \sup_{\c^* \in \mathcal{M}, \|\mathbf{c}-\mathbf{c}^*\|^2 \leq \delta} {|\tilde{P}_n-\tilde{P}|(-R_\lambda(\mathbf{c}^*,\mathbf{c}-\mathbf{c}^*,. ))}    
\eeqn
 To bound the first term in this decomposition, consider the random variable
 \beqnn
 Z_n=(\tilde{P}_n-\tilde{P}) \left\langle \mathbf{c}^*-\mathbf{c}, \nabla_{\c}\gamma_\lambda(\mathbf{c}^*,.) \right\rangle =\frac{2}{n}\sum_{u=1}^k\sum_{j=1}^d(c_{u,j}-c^*_{u,j})\sum_{i=1}^n\int_{V_u}\frac{1}{\lambda}\mathcal{K}_\eta\left(\frac{Z_i-x}{\lambda}\right)(x_j-c_{u,j})dx.
 \eeqnn
By a simple Hoeffding's inequality, $Z_n$ is a subgaussian random variable. Its variance can be bounded as follows:
\beqnn
\mathrm{var} Z_n&=&\frac{4}{n}\sum_{u=1}^k\sum_{j=1}^d(c_{u,j}-c^*_{u,j})^2\mathrm{var}\int_{V_u}\frac{1}{\lambda}\mathcal{K}_\eta\left(\frac{Z-x}{\lambda}\right)(x_j-c_{u,j})dx\\
&\leq &\frac{4}{n}\delta \E\left(\int_{V_{u^+}}\frac{1}{\lambda}\mathcal{K}_\eta\left(\frac{Z-x}{\lambda}\right)(x_j-c_{{u^+},j})dx\right)^2\\
&\leq &C\frac{4}{n}\delta \int\left|\mathcal{F}\left[\frac{1}{\lambda}\mathcal{K}_\eta\left(\frac{\cdot}{\lambda}\right)\right](t)\right|^2\left|\mathcal{F}[(\pi_j-c_{{u^+},j})\mathrm{1}_{V_{u^+}}](t)\right|^2dt\\
&\leq &C\frac{4}{n}\delta\Pi_{i=1}^d\lambda_i^{-2\beta_i} \int_{V_{u^+}}(x_j-c_{u^+,j})^2dx\\
&\leq & C\Pi_{i=1}^d\lambda_i^{-2\beta_i}\frac{4}{n}\delta,
\eeqnn
where ${u^+}=\arg\max_u\int_{V_u}\frac{1}{\lambda}\mathcal{K}_\eta\left(\frac{Z-x}{\lambda}\right)(x_j-c_{u,j})dx$ and $\pi_j:x\mapsto x_j$, and where we use the same argument as in Lemma \ref{lip} under assumption \textbf{(K1)}. We hence have using for instance a maximal inequality due to Massart \citet[Part 6.1]{saintflour}:
\beqnn
               \mathbb{E} \left (  \sup_{\c^* \in \mathcal{M}, \|\mathbf{c}-\mathbf{c}^*\|^2 \leq \delta} { (\tilde{P}_n-\tilde{P}) \left\langle \mathbf{c}^*-\mathbf{c}, \nabla_{\c}\gamma_\lambda(\mathbf{c}^*,.) \right\rangle  } \right ) 
                \leq
               C\frac{\Pi_{i=1}^d\lambda_i^{-\beta_i}}{\sqrt{n}}\sqrt{\delta}.
\eeqnn
                We obtain for the first term in (\ref{dec}) the right order. To prove that the second term in (\ref{dec}) is smaller, note that from \cite{pollard82}, we have:
\beqnn
|R_\lambda(\mathbf{c}^*,\mathbf{c}-\mathbf{c}^*,z)|&\leq&\| \c - \c^* \|^{-1} \left(\left\langle\mathbf{c}-\mathbf{c}^*,\nabla_{\c}\gamma_\lambda(\mathbf{c}^*,z)\right\rangle+\max_{j=1,\ldots k}(|\|z-\mathbf{c}_j\|^2-\|z-\mathbf{c}_j^*\|^2|\right)\\
&\leq &  \|\nabla_{\c}\gamma_\lambda(\mathbf{c}^*,z)\|+\| \c - \c^* \|^{-1}\sum_{j=1,\ldots k}|\|z-\mathbf{c}_j\|^2-\|z-\mathbf{c}_j^*\|^2|\\
&\leq &  C(\Pi_{i=1}^d\lambda_i^{-\beta_i}+\no z\no)
\eeqnn
we we use in last line:
$$
\|\nabla_{\c}\gamma_\lambda(\mathbf{c}^*,z)\|^2=4\sum_{j,k}\left(\int\frac{1}{\lambda}\mathcal{K}_\eta\left(\frac{z-x}{\lambda}\right)(x_j-c^*_{u,j})\mathbf{1}_{V^*_u}(x)dx\right)^2\leq C\Pi_{i=1}^d\lambda_i^{-2\beta_i}.
$$ 
Hence it is possible to apply a chaining argument as in \citet{levrard} to the class
$$
\mathcal{F}=\{R_\lambda(\mathbf{c}^*,\mathbf{c}-\mathbf{c}^*,\cdot),\mathbf{c}^*\in\mathcal{M},\c\in\R^{kd}:\| \c-\c^*\|\leq\sqrt{\delta}\},
$$
which has an enveloppe function $F(\cdot)\leq C(\Pi_{i=1}^d\lambda_i^{-\beta_i}+\| \cdot\|)\in L_2(\tilde{P})$ provided that $\E\|\epsilon\|^2<\infty$. We arrive at the conclusion.
\end{proof}
\bibliographystyle{plain}
         \bibliography{referencejmlr}

\begin{thebibliography}{35}
\providecommand{\natexlab}[1]{#1}
\providecommand{\url}[1]{\texttt{#1}}
\expandafter\ifx\csname urlstyle\endcsname\relax
  \providecommand{\doi}[1]{doi: #1}\else
  \providecommand{\doi}{doi: \begingroup \urlstyle{rm}\Url}\fi

\bibitem[Adamczak(2008)]{adamczak}
R.~Adamczak.
\newblock A tail inequality for suprema of unbounded empirical processes with
  applications to markov chains.
\newblock \emph{Electronic Journal of Probability}, 13 (34):\penalty0
  1000--1034, 2008.

\bibitem[Antos et~al.(2005)Antos, Gy{\"o}rfi, and Gy{\"o}rgy]{gg}
A.~Antos, L.~Gy{\"o}rfi, and A.~Gy{\"o}rgy.
\newblock Individual convergence rates in empirical vector quantizer design.
\newblock \emph{IEEE Trans. Inform. Theory}, 51 (11), 2005.

\bibitem[Bartlett and Mendelson(2006)]{empimini}
P.L. Bartlett and S.~Mendelson.
\newblock Empirical minimization.
\newblock \emph{Probability Theory and Related Fields}, 135 (3):\penalty0
  311--334, 2006.

\bibitem[Bartlett et~al.(1998)Bartlett, Linder, and Lugosi]{lbclustering}
P.L. Bartlett, T.~Linder, and G.~Lugosi.
\newblock The minimax distortion redundancy in empirical quantizer design.
\newblock \emph{IEEE Trans. Inform. Theory}, 44 (5), 1998.

\bibitem[Biau and Fisher(2012)]{fisherbiau}
G.~Biau and A.~Fisher.
\newblock Parameter selection for principal curves.
\newblock \emph{IEEE Transactions on Information Theory}, 58, 2012.

\bibitem[Biau et~al.(2008)Biau, Devroye, and Lugosi]{biau}
G.~Biau, L.~Devroye, and G.~Lugosi.
\newblock On the performance of clustering in hilbert spaces.
\newblock \emph{IEEE Transactions on Information Theory}, 54 (2), 2008.

\bibitem[Blanchard et~al.(2008)Blanchard, Bousquet, and Massart]{svm}
G.~Blanchard, O.~Bousquet, and P.~Massart.
\newblock Statistical performance of support vector machines.
\newblock \emph{The Annals of Statistics}, 36 (2):\penalty0 489--531, 2008.

\bibitem[Bousquet(2002)]{bousquet}
O.~Bousquet.
\newblock A bennet concentration inequality and its application to suprema of
  empirical processes.
\newblock \emph{C.R. Acad. SCI. Paris Ser. I Math}, 334:\penalty0 495--500,
  2002.

\bibitem[Butucea(2007)]{butucea}
C.~Butucea.
\newblock goodness-of-fit testing and quadratic functionnal estimation from
  indirect observations.
\newblock \emph{The Annals of Statistics}, 35:\penalty0 1907--1930, 2007.

\bibitem[Comte and Lacour(2012)]{comtelacour}
F.~Comte and C.~Lacour.
\newblock Anisotropic adaptive kernel deconvolution.
\newblock to appear in Annales de l'Institut Henri Poincaré, 2012.

\bibitem[Fan(1991)]{Fan}
J.~Fan.
\newblock On the optimal rates of convergence for nonparametric deconvolution
  problems.
\newblock \emph{Annals of Statistics}, 19:\penalty0 1257--1272, 1991.

\bibitem[Graf and Luschgy(2000)]{graf}
Siegfried Graf and Harald Luschgy.
\newblock \emph{Foundation of quantization for probability distributions}.
\newblock Springer-Verlag, 2000.
\newblock Lecture Notes in Mathematics, volume 1730.

\bibitem[Hartigan(1975)]{hartigan75}
J.A. Hartigan.
\newblock \emph{Clustering algorithms}.
\newblock Wiley, 1975.

\bibitem[{K\'egl} et~al.(2000){K\'egl}, Krzyzak, Linder, and K.]{kklz}
B.~{K\'egl}, A.~Krzyzak, T.~Linder, and Zeger K.
\newblock Learning and design of principal curves.
\newblock \emph{IEEE Tansactions on Pattern Analysis and Machine Intelligence},
  22:\penalty0 282--297, 2000.

\bibitem[Koltchinskii(2006)]{kolt}
V.~Koltchinskii.
\newblock Local rademacher complexities and oracle inequalties in risk
  minimization.
\newblock \emph{The Annals of Statistics}, 34 (6):\penalty0 2593--2656, 2006.

\bibitem[Koltchinskii and Panchenko(2000)]{koltpachenko}
V.~Koltchinskii and D.~Panchenko.
\newblock Rademacher processes and bounding the risk of function learning.
\newblock In \emph{High Dimensional Probability II}, pages 443--459. E. Giné,
  D. Mason and J. Wellner, eds., 2000.

\bibitem[Lecu\'e and Mendelson(2012)]{lm09}
G.~Lecu\'e and S.~Mendelson.
\newblock General non-exact oracle inequalities for classes with a
  subexponential envelope.
\newblock \emph{The Annals of Statistics}, 40 (2):\penalty0 832--860, 2012.

\bibitem[Ledoux(1996)]{ledoux}
M.~Ledoux.
\newblock On talagrand's deviation inequalities for product measures.
\newblock \emph{ESAIM, Probability and Statistics}, 1:\penalty0 63--87, 1996.

\bibitem[Levrard(2012)]{levrard}
C.~Levrard.
\newblock Fast rates for empirical vector quantization.
\newblock \emph{hal.inria.fr/hal-00664068}, 2012.

\bibitem[Linder et~al.(1994)Linder, Lugosi, and Zeger]{llz}
T.~Linder, G.~Lugosi, and K.~Zeger.
\newblock Rates of convergence in the source coding theorem, in empirical
  quantizer design, and in universal lossy source coding.
\newblock \emph{IEEE Trans. Inform. Theory}, 40 (6), 1994.

\bibitem[Loustau(2012)]{loustau12}
S.~Loustau.
\newblock Inverse statistical learning.
\newblock In revision to Electronic Journal of Statistics, 2012.

\bibitem[Loustau and Marteau(2012)]{pinkfloyds}
S.~Loustau and C.~Marteau.
\newblock Minimax fast rates for discriminant analysis with errors in
  variables.
\newblock In revision to Bernoulli, 2012.

\bibitem[Mammen and Tsybakov(1999)]{mammen}
E.~Mammen and A.B. Tsybakov.
\newblock Smooth discrimination analysis.
\newblock \emph{The Annals of Statistics}, 27 (6):\penalty0 1808--1829, 1999.

\bibitem[Massart(2000)]{massart}
P.~Massart.
\newblock About the constants in talagrand's inequality for empirical
  processes.
\newblock \emph{The Annals of Probability}, 29 (2):\penalty0 863--884, 2000.

\bibitem[Massart(2007)]{saintflour}
P.~Massart.
\newblock Concentration inequalities and model selection.
\newblock Ecole d'\'et\'e de Probabilit\'es de Saint-Flour 2003. Lecture Notes
  in Mathematics, Springer, 2007.

\bibitem[Meister(2009)]{meister}
A.~Meister.
\newblock \emph{Deconvolution problems in nonparametric statistics}.
\newblock Springer-Verlag, 2009.

\bibitem[Pollard(1981)]{pollard81}
D.~Pollard.
\newblock Strong consistency of k-means clustering.
\newblock \emph{The Annals of Statistics}, 9 (1), 1981.

\bibitem[Pollard(1982)]{pollard82}
D.~Pollard.
\newblock A central limit theorem for $k$-means clustering.
\newblock \emph{The Annals of Probability}, 10 (4), 1982.

\bibitem[Rio(2000)]{rio}
E.~Rio.
\newblock In\'egalit\'e de concentration pour les processus empiriques de
  classes de parties.
\newblock \emph{Probability Theory and Related Fields}, 119:\penalty0 163--175,
  2000.

\bibitem[Sandilya and Kulkarni(2002)]{polygonalalgorithm}
S.~Sandilya and S.R. Kulkarni.
\newblock Principal curves with bounded turn.
\newblock \emph{IEEE Tansactions on Information Theory}, 48:\penalty0
  2789--2793, 2002.

\bibitem[Tsybakov(2004{\natexlab{a}})]{booktsybakov}
A.B. Tsybakov.
\newblock \emph{Introduction \`a l'estimation non-param\'etrique}.
\newblock Springer-Verlag, 2004{\natexlab{a}}.

\bibitem[Tsybakov(2004{\natexlab{b}})]{tsybakov2004}
A.B. Tsybakov.
\newblock Optimal aggregation of classifiers in statistical learning.
\newblock \emph{The Annals of Statistics}, 32 (1):\penalty0 135--166,
  2004{\natexlab{b}}.

\bibitem[{Van De Geer}(2000)]{vdg}
S.~{Van De Geer}.
\newblock \emph{Empirical Processes in M-estimation}.
\newblock Cambridge University Press, 2000.

\bibitem[van~der Vaart and Weelner(1996)]{wvdv}
A.~W. van~der Vaart and J.~A. Weelner.
\newblock \emph{Weak convergence and Empirical Processes. With Applications to
  Statistics}.
\newblock Springer Verlag, 1996.

\bibitem[Vapnik(2000)]{vapnik2000}
V.~Vapnik.
\newblock \emph{The Nature of Statistical Learning Theory}.
\newblock Statistics for Engineering and Information Science, Springer, 2000.

\end{thebibliography}
\end{document}